\documentclass{amsart}

\title{Holomorphic geometric structures  on K{\"a}hler--Einstein manifolds}
\author{Benjamin McKay}
\address{University College Cork \\ National University of Ireland \\ Cork, Ireland}
\date{\today} 

\usepackage[T1,T2A]{fontenc}
\usepackage[utf8]{inputenc}
\usepackage[russian,english]{babel}
\usepackage{array}
\usepackage{multirow}
\usepackage{ifthen}
\usepackage{ifpdf}
\usepackage{amssymb}
\usepackage{amsfonts}
\usepackage{euscript}
\usepackage{braket}
\usepackage{varioref}
\usepackage{booktabs}
\ifpdf
\usepackage[pdftex,pagebackref]{hyperref}
\else\fi
\usepackage{tikz}
\usetikzlibrary{decorations.markings}
\usepackage{tikz-cd}
\usepackage[safe]{tipa}

\usepackage{mathabx}

\newtheorem{theorem}{Theorem}
\newtheorem{corollary}{Corollary}
\newtheorem{lemma}{Lemma}
\newtheorem{proposition}{Proposition}
\theoremstyle{remark}

\newtheorem{example}{Example}

\newcounter{remarkCounter}
\newcommand{\hook}{\ensuremath{\mathbin{ \hbox{\vrule height1.4pt
        width4pt depth-1pt \vrule height4pt width0.4pt depth-1pt}}}}

\renewcommand{\C}[1]{\ensuremath{\mathbb{C}^{#1}}}

\newcommand{\Z}[1]{\ensuremath{\mathbb{Z}^{#1}}}
\newcommand{\pr}[1]{\ensuremath{\left(#1\right)}}

\newcommand{\Sch}{\text{\upshape{\foreignlanguage{russian}{Ш}}}}

\makeatletter
\newcommand{\pd}[3][1]
{
\frac{%
\partial%
\ifnum\pdf@strcmp{#1}{1}=0\else^#1\fi%
#2}%
{%
\partial
\ifnum\pdf@strcmp{#1}{1}=0\else^#1\fi%
#3
}%
}
\makeatother

\newcommand{\GL}[1]{\ensuremath{\operatorname{GL}\!\pr{#1}}}
\newcommand{\LieGL}[1]{\ensuremath{\mathfrak{gl}\!\pr{#1}}}

\newcommand{\LieSL}[1]{\ensuremath{\mathfrak{sl}\!\pr{#1}}}

\newcommand{\PSL}[1]{\ensuremath{\operatorname{PSL}\!\pr{#1}}}
\newcommand{\Un}[1]{\operatorname{U}\!\pr{#1}}
\newcommand{\Sp}[1]{\operatorname{Sp}\!\pr{#1}}
\newcommand{\PU}[1]{\operatorname{\mathbb{P}U}\!\pr{#1}}
\newcommand{\SU}[1]{\operatorname{SU}\!\pr{#1}}
\newcommand{\PO}[1]{\operatorname{\mathbb{P}O}\!\pr{#1}}
\newcommand{\SO}[1]{\operatorname{SO}\!\pr{#1}}
\renewcommand{\Proj}[1]{\mathbb{P}^{#1}}
\DeclareMathOperator{\Ad}{Ad}
\DeclareMathOperator{\ad}{ad}
\newcommand{\Lm}[2]{\ensuremath{\Lambda^{#1}\pr{#2}}}
\newcommand{\nForms}[2]{\ensuremath{\Omega^{#1}\pr{#2}}}
\newcommand{\Cohom}[2]{\ensuremath{H^{#1}\!\pr{{#2}}}}
\newcommand{\Gr}[2]{\ensuremath{\operatorname{Gr}\pr{{#1},{#2}}}}

\newcommand{\LieG}{\ensuremath{\mathfrak{g}}}
\newcommand{\LieH}{\ensuremath{\mathfrak{h}}}
\newcommand{\LieK}{\ensuremath{\mathfrak{h}}}
\newcommand{\LieP}{\ensuremath{\mathfrak{p}}}

\newcommand{\LieB}{\ensuremath{\mathfrak{b}}}

\newcommand{\Aut}[1]{\ensuremath{\operatorname{Aut} #1}}
\newcommand{\Sym}[2]{\ensuremath{\operatorname{Sym}^{#1}\pr{#2}}}
\newcommand{\torop}{\ensuremath{D}}
\newcommand{\tortr}{\ensuremath{\operatorname{tr}}}

\newcommand{\bari}{\bar{\imath}}
\newcommand{\barj}{\bar{\jmath}}
\newcommand{\frm}[1]{\ensuremath{F#1}}

\makeatletter
\newcommand{\Chern}[2][]
{
\ifnum\pdf@strcmp{#1}{}=0{c_{#2}}\else{c_{#1}\!\pr{#2}}\fi%
}
\makeatother

\begin{document}
\begin{abstract}
We prove that the compact K\"ahler manifolds with \(\Chern{1} \ge 0\) that admit holomorphic parabolic geometries are the flat bundles of rational homogeneous varieties over complex tori.
We also prove that the compact K\"ahler manifolds with \(\Chern{1} < 0\) that admit holomorphic cominiscule geometries are the locally Hermitian symmetric varieties.
\end{abstract}
\maketitle \tableofcontents

\section{Introduction}
We will classify all holomorphic parabolic geometries on compact K\"ahler manifolds with \(\Chern{1} \ge 0\); these turn out to be constructed from flat bundles on complex tori.
This motivates the problem of classification for \(\Chern{1}<0\); we will only solve part of this problem. 
Every locally Hermitian symmetric variety has a holomorphic cominiscule geometry.
We prove that a compact complex manifold \(M\) with \(\Chern{1}<0\) admits a holomorphic cominiscule geometry just when \(M\) is a locally Hermitian symmetric variety, and we prove that the moduli space of cominiscule geometries on \(M\) is a finite dimensional complex vector space.
Moreover, we find that the normal cominiscule geometries on locally Hermitian symmetric varieties are precisely the obvious flat ones.

There is a subtlety in the study of holomorphic Cartan geometries because we need to be careful about holomorphic moduli of principal bundles, an issue which does not arise in the real smooth category where these geometric structures have been thoroughly studied.
Principal bundles of holomorphic Cartan geometries with a fixed model and defined on a fixed complex manifold can vary in the moduli space of holomorphic principal bundles \cite{McKay2011}.
We prove that cominiscule geometries with fixed underlying first order structure all have isomorphic principal bundles.
We prove that any holomorphic parabolic geometry with reducible model on any smooth projective variety gives that variety a holomorphic local product structure.

The classification of normal holomorphic cominiscule geometries on K\"ahler--Einstein manifolds is known for geometries with certain models \cite{Kobayashi/Ochiai:1980,Kobayashi/Ochiai:1981,Kobayashi/Ochiai:1981b,Kobayashi/Ochiai:1982}.
Similarly, the rigidity of \emph{flat} holomorphic cominiscule geometries with irreducible models on locally Hermitian symmetric varieties is known \cite{Klingler:2001}.
We largely follow the methods of those papers.
After this paper was released on the arXiv, Antonio Di Scala pointed out that it is similar to \cite{Catanese/DiScala:2013}, which had appeared earlier on the arXiv.

\section{Definitions}

\subsection{Cartan geometries}
Sharpe \cite{Sharpe:2002} gives an introduction to Cartan geometries.
If \(E \to M\) is a principal right \(G\)-bundle, we will write the right \(G\)-action as \(r_g e = eg\).
Suppose that \(M\) is a manifold, \(V_M \subset TM\) a vector bundle and \(\pi \colon E \to M\) a principal bundle.
Let 
\[
V_E = \Set{(e,v)|e \in E, v \in T_e E, \pi'(e)v \in V_M}.
\] 

Let \(H \subset G\) be a closed subgroup of a Lie group, with Lie algebras \(\LieH \subset \LieG\) and let \(X=G/H\). 
A \((G,X)\)-geometry, or \(G/H\)-geometry, or \emph{Cartan geometry} modelled on \((G,X)\), on a vector subbundle \(V_M \subset TM\) on a manifold \(M\) is a choice of principal right \(H\)-bundle \(E \to M\), and section \(\omega\) of \(V_E^* \otimes \LieG\), called the \emph{Cartan connection}, which satisifies all of the following conditions:
\begin{enumerate}
\item
\(
r_h^* \omega = \Ad_h^{-1} \omega
\) for all \(h \in H\).
\item
\(\omega_e \colon V_e E \to \LieG\) is a linear isomorphism at each point 
\(e \in E\).
\item
For each \(A \in \LieG\), define a vector field \(\vec{A}\) on \(E\) by
the equation \(\vec{A} \hook \omega = A\);
the vector fields \(\vec{A}\) for \(A \in \LieH\) generate the right \(H\)-action.
\end{enumerate}

A \emph{Cartan geometry on a manifold} \(M\) means a Cartan geometry on \(V_M=TM\).
Similarly a Cartan geometry on a nonsingular foliation \(F\) means a Cartan geometry on the tangent bundle of \(F\).
For example, the bundle \(G \to X=G/H\) is the total space of a Cartan geometry on \(X\), with Cartan connection \(\omega=g^{-1} \, dg\) the left invariant Maurer--Cartan 1-form on \(G\); this geometry is called the \emph{model Cartan geometry}.

If \(E_0 \to M_0\) and \(E_1 \to M_1\) are Cartan geometries with models \(\pr{G_0,X_0}\) and \(\pr{G_1,X_1}\) and Cartan connections \(\omega_0\) and \(\omega_1\), their \emph{product geometry} is the Cartan geometry with bundle \(E_0 \times E_1 \to M_0 \times M_1\) and Cartan connection \(\omega_0+\omega_1\).

Suppose that \(\pr{G,X}\) is a homogeneous space and let \(H\) be the stabilizer of some point of \(X\).
Suppose that \(X=\prod_i X_i\) splits \(G\)-equivariantly and
let \(K_i \subset G\) be the subgroup of \(G\) acting trivially on \(X_i\).
Let \(G_i=G/K_i\) be the quotient of \(G\) acting effectively on \(X_i\) and \(H_i \subset G_i\) the stabilizer of the associated point of \(X_i\).
Pick a \(\pr{G,X}\)-geometry \(\pi \colon E \to M\) on a subbundle \(V=V_M\) on some manifold \(M\).
For each point \(e \in E\) over any point \(m \in M\),  the Cartan connection \(\omega\) induces a map \(\omega + \LieH \colon T_m M \to \LieG/\LieH\).
The splitting \(X=\prod_i X_i\) defines a splitting \(V=\bigoplus_i V_i\) by the requirement that \(\omega+ \LieH\) take \(V_i\) to \(\LieG_i/\LieH_i\).
Let \(E_i = E/K_i\) and \(\omega_i=\omega+\LieK_i\): then \(E_i \to M\) is a \(\pr{G_i,X_i}\)-geometry on \(V_i\).

On the other hand, suppose that various Lie groups \(G_i\) each act smoothly, transitively and effectively on various homogeneous spaces \(X_i\) and we let \(X=\prod X_i\) and \(G=\prod_i G_i\). 
Suppose that \(M\) bears vector bundles \(V_i \subset TM\) and \(V_i \cap V_j = \left\{0\right\}\) when \(i \ne j\) and each \(V_i\) has a \(\pr{G_i,X_i}\)-geometry \(\pi_i \colon E_i \to M\). 
Let \(V=\bigoplus V_i \subset TM\) and let \(E\) be the product bundle of the \(E_i \to M\).
Add up the Cartan connections to produce a \(\pr{G,X}\)-geometry on \(V,\) the \emph{product geometry}.

\subsection{Curvature}

If \(E \to M\) is a Cartan geometry over a vector bundle \(V_M \subset TM\), let \(I=V_M^{\perp} \subset T^*M\) and form the sheaf \(I'\) of 2-forms which are locally divisible either by local sections of \(I\) or by exterior derivatives of local sections of \(I\).
Locally extend \(\omega\) to an \(H\)-equivariant 1-form valued in \(\LieG\);
the \emph{curvature} of \(E \to M\) is \(d \omega + \frac{1}{2}\left[\omega,\omega\right] \in \pr{\nForms{2}{M}/I'}  \otimes_H \LieG\). 
In particular, the curvature of a Cartan geometry on a foliation is just the usual notion of curvature of a Cartan geometry on the leaves of the foliation.

\subsection{First order structures}

Suppose that \(V_0\) is a vector space.
The \(V_0\)-valued \emph{frame bundle} \(\frm{V}\) of a vector subbundle \(V=V_M \subset TM\) on a manifold \(M\) is the associated principal bundle, i.e. the set of all pairs \((m,u)\) with \(m \in M\) and \(u \colon V_m \to V_0\) a linear isomorphism. 
Let \(\pi \colon \frm{V} \to M\) be the map \((m,u) \to m\).
The frame bundle is a principal right \(\GL{V_0}\)-bundle under the obvious action \(r_g (m,u) = \left(m, g^{-1} u\right)\).
Let \(V_{\frm{V_M}} = \pr{\pi'}^{-1}V_M\). 
Define a section \(\sigma\) of \(V^*_{\frm{V_M}} \otimes V_0\) on \(\frm{V_M}\) by
\(
v \hook \sigma = u\left(\pi'(m,u)v\right),
\)
where \(v \in V_{(m,u)} \frm{V_M}\).

Suppose that \(G \to \GL{V_0}\) is a morphism of Lie groups and that \(V_M \subset TM\) is a  subbundle.
A \emph{\(G\)-structure} or \emph{first order structure} on \(V_M\) is a principal right \(G\)-bundle \(B \to M\) together with a \(G\)-equivariant bundle map \(B \to \frm{V_M}\).
If \(G_0\) is the kernel of the morphism \(G \to \GL{V_0}\), then \(B/G_0 \to \frm{V}\) is the \emph{underlying immersed first order structure}, and is a submanifold of \(\frm{V_M}\).

\subsection{Underlying first order structures}

Let \(X=G/H\) be a homogeneous space.
Let \(M\) be a manifold with a vector bundle \(V_M \subset TM\), and \(\pi \colon E \to M\) a \((G,X)\)-geometry on \(V_M\).
Let \(\LieG\) and \(\LieH\) be the Lie algebras of \(G\) and \(H\). 
Let \(V_0=\LieG/\LieH\). 
Let \(\frm{V_M}\) be the \(V_0\)-valued frame bundle. 
Let \(\sigma=\omega+\LieH\), a semibasic 1-form defined on \(V_E\) and valued in \(V_0\). 
At each point \(e \in E\) the 1-form \(\sigma\) determines a linear isomorphism \(u \colon V_{M,m} \to V_0\) by the equation \(u(\pi'(e)v) = v \hook \sigma\). 
Map \(e \in E \mapsto u=u(e) \in \frm{V_M}\). 
This map is an \(H\)-structure. 
The fibers of this map consist of the orbits in \(E\) of the subgroup \(H_1\) of \(H\) acting trivially on \(V_0=\LieG/\LieH\).
The map descends to a map \(E/H_1 \to \frm{V_M}\) called the \emph{underlying \(\left(H/H_1\right)\)-structure} or \emph{underlying first order structure} of the \(G/H\)-geometry.

\subsection{Holomorphy}

From now on, all Cartan geometries in this paper will be assumed to be holomorphic Cartan geometries, i.e. \(G\) is a complex Lie group, \(X\) is a complex manifold on which \(G\) acts holomorphically, effectively and transitively, \(E \to M\) is a holomorphic principal bundle on a complex manifold \(M\), \(V_M \subset M\) is a holomorphic vector bundle, and the Cartan connection is holomorphic.
Isomorphisms of Cartan geometries will therefore be biholomorphic.

\subsection{Rational homogeneous varieties and parabolic geometries}

A \emph{rational homogeneous variety} is a pair \((G,X)\) where \(G\) is a complex semisimple Lie group acting effectively, transitively and holomorphically on a compact complex manifold \(X\). 
A \emph{parabolic geometry} is a Cartan geometry modelled on a rational homogeneous variety.
By a theorem of Chevalley \cite{Borel:1991} theorem 11.16 p. 154, if \(G\) is a complex semisimple Lie group, then every parabolic subgroup \(P \subset G\) (i.e. containing a maximal solvable subgroup) is connected and equal to its normalizer in \(G\) and has compact quotient \(X=G/P\) and every rational homogeneous variety arises this way.
The center of \(G\) acts trivially on \(X\), so without loss of generality \(G\) is in adjoint form.

If \(\pr{G,X}\) is a rational homogeneous variety then, up to reordering factors, there is a unique decomposition \(G=\prod_i G_i\) and \(X=\prod_i X_i\) into products so that each \(G_i\) acts trivially on all \(X_j\) except \(X_i\) and \(\pr{G_i,X_i}\) is a rational homogeneous variety which cannot be thus decomposed, i.e. is irreducible.

\subsection{Cominiscule varieties}

A \emph{cominiscule variety} is a rational homogeneous variety \((G,X)\) so that the stabilizer \(P\) of a point \(x_0 \in X\) acts on \(T_{x_0} X\) as a sum of irreducible representations; see \cite{Baston/Eastwood:1989,Landsberg:2005}.
Each cominiscule variety \((G,X)\) admits a K\"ahler metric, invariant under a maximal compact subgroup of \(G\), in which \(X\) becomes a compact Hermitian symmetric space.
The group \(G\) in our definition is a complex group of biholomorphisms of \(M=G/P\); \(G\) does \emph{not} preserve any metric on \(G/P\).
For example, if \(G/P=\Proj{n}\), we will have \(G=\PSL{n+1,\C{}}\), \emph{not} \(\PU{n+1}\). 
For this reason, we use the term \emph{cominiscule variety} rather than \emph{compact Hermitian symmetric space}.

\newcommand{\dynkinradius}{.04cm}
\newcommand{\dynkinstep}{.35cm}
\newcommand{\dynkindot}[2]{\fill (\dynkinstep*#1,\dynkinstep*#2) circle (\dynkinradius);}
\newcommand{\dynkinXsize}{1.5}
\newcommand{\dynkincross}[2]{
\draw[thick] (#1*\dynkinstep-\dynkinXsize,#2*\dynkinstep-\dynkinXsize) -- (#1*\dynkinstep+\dynkinXsize,#2*\dynkinstep+\dynkinXsize);
\draw[thick] (#1*\dynkinstep-\dynkinXsize,#2*\dynkinstep+\dynkinXsize) -- (#1*\dynkinstep+\dynkinXsize,#2*\dynkinstep-\dynkinXsize);
}
\newcommand{\dynkinline}[4]{\draw[thin] (\dynkinstep*#1,\dynkinstep*#2) -- (\dynkinstep*#3,\dynkinstep*#4);}
\newcommand{\dynkindots}[4]{\draw[dotted] (\dynkinstep*#1,\dynkinstep*#2) -- (\dynkinstep*#3,\dynkinstep*#4);}
\newcommand{\dynkindoubleline}[4]{\draw[double,postaction={decorate}] (\dynkinstep*#1,\dynkinstep*#2) -- (\dynkinstep*#3,\dynkinstep*#4);}

\newenvironment{dynkin}{\begin{tikzpicture}[decoration={markings,mark=at position 0.7 with {\arrow{>}}}]}
{\end{tikzpicture}}

\begin{figure}
\renewcommand*{\arraystretch}{1.5}
\begin{tabular}{>{\(}r<{\)}m{2cm}m{8cm}}
A_n &
  \begin{dynkin}
	\dynkinline{1}{0}{2}{0};
    \dynkindots{2}{0}{3}{0};
    \dynkinline{3}{0}{5}{0};
    \dynkindots{5}{0}{6}{0};
    \dynkinline{6}{0}{7}{0};
    \foreach \x in {1,...,7}
    {
       \ifnum \x=4 {\dynkincross{\x}{0}}
       \else {\dynkindot{\x}{0}}
       \fi
    }
  \end{dynkin}
&
Grassmannian of \(k\)-planes in \(\C{n+1}\) \\
B_n &
  \begin{dynkin}
	\dynkinline{1}{0}{2}{0};
    \dynkindots{2}{0}{3}{0};
    \dynkinline{3}{0}{4}{0};
    \dynkindoubleline{4}{0}{5}{0};
    \dynkincross{1}{0};
    \foreach \x in {2,...,5}
    {
        \dynkindot{\x}{0}
    }
  \end{dynkin}
& 
\((2n-1)\)-dimensional hyperquadric, i.e. the variety of null lines in \(\C{2n+1}\)
\\
C_n 
&
  \begin{dynkin}
	\dynkinline{1}{0}{2}{0};
    \dynkindots{2}{0}{3}{0};
    \dynkinline{3}{0}{4}{0};
    \dynkindoubleline{5}{0}{4}{0};
    \dynkincross{5}{0};
    \foreach \x in {1,...,4}
    {
        \dynkindot{\x}{0}
    }
  \end{dynkin}
& 
space of Lagrangian \(n\)-planes in \(\C{2n}\)
\\
D_n
&
  \begin{dynkin}
    \foreach \x in {2,...,4}
    {
        \dynkindot{\x}{0}
    }
    \dynkindot{4.5}{.9}
    \dynkindot{4.5}{-.9}
    \dynkincross{1}{0}
    \dynkinline{1}{0}{2}{0}
    \dynkindots{2}{0}{3}{0}
    \dynkinline{3}{0}{4}{0}
    \dynkinline{4}{0}{4.5}{.9}
    \dynkinline{4}{0}{4.5}{-.9}
  \end{dynkin}
& \((2n-2)\)-dimensional hyperquadric, i.e. the variety of null lines in \(\C{2n}\)
\\
D_n
&
  \begin{dynkin}
    \foreach \x in {1,...,4}
    {
        \dynkindot{\x}{0}
    }
    \dynkincross{4.5}{.9}
    \dynkindot{4.5}{-.9}
    \dynkinline{1}{0}{2}{0}
    \dynkindots{2}{0}{3}{0}
    \dynkinline{3}{0}{4}{0}
    \dynkinline{4}{0}{4.5}{.9}
    \dynkinline{4}{0}{4.5}{-.9}
  \end{dynkin}
&
one component of the variety of maximal dimension null subspaces of \(\C{2n}\)
\\
D_n
&
  \begin{dynkin}
    \foreach \x in {1,...,4}
    {
        \dynkindot{\x}{0}
    }
    \dynkincross{4.5}{-.9}
    \dynkindot{4.5}{.9}
    \dynkinline{1}{0}{2}{0}
    \dynkindots{2}{0}{3}{0}
    \dynkinline{3}{0}{4}{0}
    \dynkinline{4}{0}{4.5}{.9}
    \dynkinline{4}{0}{4.5}{-.9}
  \end{dynkin}
&
the other component
\\
E_6 
&
  \begin{dynkin}
    \foreach \x in {2,...,5}
    {
        \dynkindot{\x}{0}
    }
    \dynkincross{1}{0}
    \dynkindot{3}{1}
    \dynkinline{1}{0}{5}{0}
    \dynkinline{3}{0}{3}{1}
  \end{dynkin}
&
complexified octave projective plane
\\
E_6 
&
  \begin{dynkin}
    \foreach \x in {1,...,4}
    {
        \dynkindot{\x}{0}
    }
    \dynkincross{5}{0}
    \dynkindot{3}{1}
    \dynkinline{1}{0}{5}{0}
    \dynkinline{3}{0}{3}{1}
  \end{dynkin}
&
its dual plane
\\
E_7
&
  \begin{dynkin}
    \foreach \x in {1,...,5}
    {
        \dynkindot{\x}{0}
    }
    \dynkincross{6}{0}
    \dynkindot{3}{1}
    \dynkinline{1}{0}{6}{0}
    \dynkinline{3}{0}{3}{1}
  \end{dynkin}
&
the space of null octave 3-planes in octave 6-space
\end{tabular}
\caption{The classification of irreducible cominiscule varieties; every cominiscule variety is a product of these.}
\end{figure}

\begin{example}\label{example:dual}
Suppose that \(X=G/P\) is a cominiscule variety.
Denote by \(X'=G'/P'\) its dual noncompact Hermitian symmetric space \cite{Helgason:1978}, i.e. \(G' \subset G\) is a real form acting on \(X\), and \(P'=P\cap G'\) is compact, and \(G'\) acts on \(X\) with an open orbit \(X'=G'/P'\). 
Pullback the standard flat \((G,X)\)-geometry on \(X\) to a flat \((G,X)\)-geometry on \(X'\). 
If \(\Gamma \subset G'\) is a cocompact lattice, then the flat \((G,X)\)-geometry on \(X'\) is pulled back from a unique flat \((G,X)\)-geometry on the manifold \(M=\Gamma \backslash X'\), called the \emph{standard geometry} on \(M\).
\end{example}

\begin{example}
If \(\pr{G,X}=\pr{\PSL{n+1,\C{}},\Proj{n}}\) then any holomorphic \(\pr{G,X}\)-geometry is called a \emph{holomorphic projective connection}.
The space of holomorphic projective connections on any Riemann surface is canonically identified with the space of holomorphic quadratic differentials \cite{Loray/MarinPerez:2009}.
\end{example}

\begin{example}
Call a holomorphic cominiscule geometry \emph{standard} if, after perhaps replacing by the pull back to a finite \'etale covering space, it becomes a product geometry of the form \(M=M' \times \prod_j M_j\) is a product, where each \(M_j\) is a compact Riemann surface with a holomorphic projective connection and \(M'=\Gamma \backslash X'\) is the standard geometry, for \(X'\) the noncompact dual of a cominiscule variety.
\end{example}

\begin{theorem}\label{theorem:main}
If \(M\) is a connected compact complex manifold with \(\Chern[1]{M} < 0\) and bearing a holomorphic cominiscule geometry, then \(M\) admits a unique holomorphic normal cominiscule geometry, and moreover that holomorphic normal cominiscule geometry is standard.
In particular, \(M\) is a locally Hermitian symmetric variety.
\end{theorem}
Theorem~\ref{theorem:main} is a consequence of theorem~\vref{theorem:big}.

\subsection{Roots}

Helgason \cite{Helgason:1978} and Knapp \cite{Knapp:2002} both prove  all of the results we will use on Lie algebras; we give some definitions only to fix notation; also see Landsberg \cite{Landsberg:2005}.
Suppose that \(G/P\) is a cominiscule variety.
Fix a Cartan subalgebra of \(\LieG\).
Let \(\LieB\) be the Borel subalgebra, i.e. the subalgebra consisting of the sum of the Cartan subalgebra and all of the positive root spaces. 
The Lie algebra \(\LieP\) can be made to contain \(\LieB\) by conjugation \cite{Fulton/Harris:1991} p. 382; assume without loss of generality from now on that \(\LieB \subset \LieP\).
For each root \(\alpha\), let \(\LieG_{\alpha}\) be the root space of the root \(\alpha\). 
Write \(\LieG\) as a sum of simple Lie algebras \(\LieG = \bigoplus_i \LieG^i\).
Let \(\LieP^i = \LieG^i \cap \LieP\).
Then \(\LieP = \bigoplus_i \LieP^i\).
A root \(\alpha\) is \emph{compact} if both \(\LieG_{\alpha} \cap \LieP \ne \left\{0\right\}\) and \(\LieG_{-\alpha} \cap \LieP \ne \left\{0\right\}\).
Each \(\LieG^i\) has a unique noncompact simple root \(\alpha^i\).
Write a noncompact root of \(\LieG\) as \(\beta^+\) or \(\beta^-\) to indicate that it is a noncompact positive or noncompact negative root. 
If we write a root without a superscript, e.g. \(\beta\), this means that it is allowed to be any of the above, or even a compact root.
If \(\beta^-\) is any (necessarily noncompact negative) root, we will write \(\beta^+\) to mean \(-\beta^-\), etc.

\subsection{The grading}

Grade \(\LieG\) into a sum:
\(
\LieG = \LieG_- \oplus \LieG_0 \oplus \LieG_+,
\)
where \(\LieG_-\) is the sum of all noncompact negative root spaces, \(\LieG_{+}\) is the sum of all noncompact positive root spaces, and \(\LieG_0\) is the sum of the Cartan subalgebra of \(\LieG\) with all of the compact root spaces, so
\(
\LieP = \LieG_0 \oplus \LieG_+.
\)
Each of the Lie algebras \(\LieG_-, \LieG_0, \LieG_+\) is the Lie algebra of a connected algebraic Lie subgroup \(G_-, G_0, G_+\) of \(G\). 
The subgroups \(G_-\) and \(G_+\) are abelian and simply connected (see Knapp \cite{Knapp:2002} theorem 7.129 p. 506).
This grading of \(\LieG\) is only \(G_0\)-invariant, but the associated filtration is \(P\)-invariant and \(P=G_0G_+\).
Identify
\(
\LieG_-^* = \LieG_+
\)
using the Killing form.
If \(X_+ \in \LieG_+\) and \(Y=Y_-+Y_0+Y_+ \in \LieG\) then
\begin{equation}\label{equation:AdexpX}
\Ad\left(e^{X_+}\right)Y = Y_- + \left(Y_0-\left[X_+,Y_-\right]\right) + Z_+
\end{equation}
where
\[
Z_+ = 
Y_+ - \left[X_+,Y_0\right] + \frac{1}{2}\left[X_+,\left[X_+,Y_-\right]\right].
\]

\subsection{Chevalley bases}\label{subsubsection:ChevalleyBases}

A \emph{Chevalley basis} \(X_{\alpha}, H_{\alpha}\) (see Serre \cite{Serre:2001}) is a basis
of \(\LieG\) parameterized by roots \(\alpha \in \LieH^*\) (with \(\LieH \subset \LieG\) a Cartan
subalgebra) for which
\begin{enumerate}
\item \(\left[H,X_{\alpha}\right]=\alpha(H) X_{\alpha}\)
for each \(H \in \LieH\)
\item \(\alpha\left(H_{\beta}\right)=
2 \frac{\left<\alpha, \beta\right>}{\left<\beta,\beta\right>}\)
(measuring inner products via the Killing form)
\item \(\left[H_{\alpha},H_{\beta}\right]=0\),
\item
\[
\left[X_{\alpha},X_{\beta}\right]=
\begin{cases}
H_{\alpha}, & \text{ if } \alpha+\beta=0, \\
N_{\alpha\beta} X_{\alpha+\beta}, & \text{otherwise} \\
\end{cases}
\]
with
\begin{enumerate}
\item \(N_{\alpha\beta}\) an integer,
\item \(N_{-\alpha,-\beta}=-N_{\alpha\beta}\),
\item If \(\alpha, \beta, \) and \(\alpha+\beta\) are roots, then
\(N_{\alpha\beta}=\pm (p+1)\),
where \(p\) is the largest integer for which \(\beta-p \, \alpha\) is a root,
\item
\(N_{\alpha \beta}=0\) if \(\alpha+\beta = 0\) or
if any of \(\alpha, \beta,\) or  \(\alpha+\beta\) is not a root.
\end{enumerate}
\end{enumerate}

\section{%
\texorpdfstring%
{Classification of holomorphic parabolic geometries with $\Chern{1} \ge 0$}%
{Classification of holomorphic parabolic geometries with first Chern class nonnegative}%
}

\begin{example}\label{example:tori}
Suppose that \(\bar{P} \subset \bar{Q}\) is a parabolic subgroup of a complex semisimple Lie group \(\bar{G}\) and let \(\bar{X}=\bar{G}/\bar{P}\).
Denote the Lie algebra of \(\bar{P}\) as \(\bar\LieP\).
Take a complex linear \(\bar\LieP\)-complementary subspace \(V \subset \bar\LieG\), letting \(s \in V^* \otimes \LieG\) be the inclusion map. 
Take a lattice \(\Lambda \subset V\), let \(\bar{M}=V/\Lambda\), and let \(\bar{E} = \pr{V/\Lambda} \times \bar{P}\) with elements written as \(\pr{z+\Lambda,\bar{p}}\).
Let \(\bar\omega=\bar{p}^{-1} \, d\bar{p} + \Ad\pr{\bar{p}}^{-1} s\).
Then \(\bar{E} \to \bar{M}\) is a translation invariant holomorphic \(\pr{G,X}\)-geometry on a complex torus, with holomorphic Cartan connection \(\bar{\omega}\).
\end{example}

\begin{example}\label{example:c1.positive}
Continuing with the previous example, suppose that \(G_0\) is a complex semisimple Lie group and pick parabolic subgroups \(Q_0 \subset G_0\) and \(\bar{Q} \subset \bar{P}\).
Let \(G=G_0 \times \bar{G}, P=Q_0 \times \bar{Q}\) and \(X=G/P\).
Take any group morphism \(\rho \colon \Lambda \to G_0\).
Let \(E'=V \times_{\rho} G_0 \to \bar{M}\) be the associated flat \(G_0\)-bundle.
Let \(\omega_0\) be the holomorphic flat connection on \(E' \to \bar{M}\).
Let \(E = E' \times_{\bar{M}} \bar{E}\), with obvious \(G_0 \times \bar{P}\)-action.
Write the pullback form of \(\omega_0\) to \(E\) also as \(\omega_0\) and similarly write the pullback form of \(\bar{\omega}\) to \(E\) also as \(\bar{\omega}\).
Let \(\omega=\omega_0 + \bar{\omega}\) and let \(M = E/P\).
Then \(E \to M\) is a holomorphic \(\pr{G,X}\)-geometry with Cartan connection \(\omega\).
Moreover, \(M \to \bar{M}\) is a holomorphic fiber bundle with all fibers biholomorphic to \(\pr{G_0/Q_0} \times \pr{\bar{P}/\bar{Q}}\).
In particular, \(\Chern[1]{M} \ge 0\).
\end{example}

\begin{theorem}
Suppose that \(M\) is a connected compact K\"ahler manifold with \(\Chern{1} \ge 0\) and \(\pr{G,X}\) is a rational homogeneous variety with \(\dim M = \dim X\).
Then \(M\) admits a holomorphic \(\pr{G,X}\)-geometry if and only if, after perhaps replacing \(M\) by a finite \'etale covering space,
\begin{enumerate}
\item \(M=X\) with the standard flat \(\pr{G,X}\)-geometry or
\item \(M\) is a complex torus and the holomorphic \(\pr{G,X}\)-geometry is as constructed in example~\vref{example:tori} or
\item
\(M\), with its holomorphic \(\pr{G,X}\)-geometry, is constructed as in example~\vref{example:c1.positive}.
\end{enumerate}
\end{theorem}
\begin{proof}
By a theorem of Campana, Demailly and Peternell \cite{Campana/Demailly/Peternell:2013}, every compact K\"ahler manifold \(M\) with \(\Chern{1} \ge 0\) is, up to a finite \'etale covering space, a holomorphic fiber bundle, with rationally connected fibers, over a compact K\"ahler base with \(\Chern{1}=0\), so without loss of generality assume that \(M\) is such a bundle \(M \to M'\).

If \(M\) also admits a holomorphic parabolic geometry, say modelled on a rational homogeneous variety \(\pr{G,X}\), \(X=G/P\), then \(M\) is a holomorphic fiber bundle with rational homogeneous fibers \(M \to \bar{M}\) over a compact K\"ahler manifold \(\bar{M}\) which contains no rational curves, and the holomorphic \(\pr{G,X}\)-geometry on \(M\) drops to a holomorphic parabolic geometry on \(\bar{M}\) \cite{Biswas/McKay:2010a} theorem 2 p. 4.
This means precisely that there is a surjective morphism of complex semisimple Lie groups \(G \to \bar{G}\), say with kernel \(G_0\), with both \(G\) and \(\bar{G}\) in adjoint form, carrying \(P\) to lie inside a complex parabolic subgroup \(P \to \bar{P} \subset \bar{G}\), and the bundle \(P \to E \to M\) of the parabolic geometry maps equivariantly \(E \to \bar{E}\) to a holomorphic \(\bar{P}\)-bundle \(\bar{P} \to \bar{E} \to \bar{M}\), and there is a holomorphic Cartan connection \(\bar{\omega}\) on \(\bar{E}\) which pulls back to \(\omega+\LieG_0\) and \(\omega\) is a Cartan connection for \(E \to \bar{M}\).

Since all rational curves in \(M\) live in the fibers of \(M \to \bar{M}\), the map \(M \to \bar{M}\) drops to a holomorphic fiber bundle mapping \(M' \to \bar{M}\).
The fibers of \(M'\to\bar{M}\) have preimages in \(M\) which are the fibers of \(M \to \bar{M}\), and so are rationally connected.
Any immersed rational curve \(C \subset F\) in a fiber \(F \subset M'\to\bar{M}\) has preimage a smooth subvariety of \(M\) which lies in a fiber of \(M \to \bar{M}\), so there is a copy of this variety in every fiber of \(M\to\bar{M}\), so this variety has deformations moving it around in \(M\), and so its images in \(M'\) allow that rational curve to deform freely, i.e. the ambient tangent bundle \(\left.TM'\right|_C\) is a sum of line bundles of nonnegative degree, one of which, \(TC\), has positive degree.
But the sum of these degrees must be \(\Chern[1]{M'}=0\), a contradiction unless there are no rational curves in \(M'\), i.e. \(M'=\bar{M}\).

Since \(\Chern[1]{\bar{M}}=0\) and \(\bar{M}\) admits a holomorphic parabolic geometry, it follows \cite{McKay:2011} that, after perhaps replacing \(\bar{M}\) by a finite \'etale covering space (and thereby replacing \(M\) by the pullback bundle), \(\bar{M}\) is a complex torus.
Moreover the parabolic geometry on \(\bar{M}\) is obtained by the process given in example~\vref{example:c1.positive} \cite{McKay:2011}.

Because \(G\) is semisimple, the Lie algebra \(\LieG_0\) of \(G_0\) is a sum of irreducible \(G\)-modules, i.e. of simple ideals of \(G\), and we can inductively replace any expression of \(\LieG\) as a sum of simple ideals with an expression in which the ideals are drawn from \(\LieG_0\) or from a \(\LieG_0\)-complementary \(G\)-invariant sum of simple ideals, which is then isomorphic to \(\bar{\LieG}\), so \(\LieG=\LieG_0 \oplus \bar\LieG\).
Since \(G\) and \(\bar{G}\) are in adjoint form, so is \(G_0\), and therefore \(G=G_0 \times \bar{G}\).
We need to have \(P\) contained in the preimage of \(\bar{P}\), i.e. in \(G_0 \times \bar{P}\), and to be parabolic, so \(P=Q_0 \times \bar{Q}\) for some parabolic subgroups \(Q_0 \subset G_0\) and \(\bar{Q} \subset \bar{P}\).
Since \(\omega\) is a Cartan connection for the principal bundle \(G_0 \times \bar{P} \to E \to \bar{M}\), \(\omega_0\) is a flat connection on the bundle \(G_0 \to E'=E/\bar{P} \to \bar{M}\).

Clearly we can recover the principal bundle \(E \to \bar{M}\) from the two principal bundles \(E' \to \bar{M}\) and \(\bar{E} \to \bar{M}\) as \(E=E' \times_{\bar{M}} \bar{E}\).
We can recover the Cartan connection on \(E\) by pulling back the flat connection from \(E'\) and the Cartan connection from \(\bar{E}\), and with obvious \(G_0 \times \bar{P}\)-action.
We can recover \(M\) as \(M=E/\pr{Q_0 \times \bar{Q}}\).
Therefore we can just pick any holomorphic principal \(G_0\)-bundle \(E' \to \bar{M}\) with any flat holomorphic connection, and pick any holomorphic parabolic geometry on \(\bar{M}\), and build a complex manifold \(M\) and its parabolic geometry as above.

Holomorphic \(G_0\)-bundles \(E' \to \bar{M}\) with holomorphic flat connection are classified up to isomorphism by conjugacy classes of representations of the fundamental group, i.e. group morphisms \(\Lambda \to G_0\), so the resulting geometry is constructed as in example~\vref{example:c1.positive}.
\end{proof}

\section{Invariant tensors on cominiscule varieties}

\subsection{The fundamental tensor of a cominiscule variety}
There is a well known invariant holomorphic tensor on any cominiscule variety \(G/P\); see Kobayashi and Ochiai \cite{Kobayashi/Ochiai:1981b}. 
Take any elements \(X_- \in \LieG_-\) and \(\xi^- \in \LieG_-^*\). 
Pick out the element \(Y_+ \in \LieG_+\) dual to \(\xi^-\) under the Killing form of \(\LieG\).
Define \(\tau\left(X_-,\xi^-\right) = \ad \left[X_-,Y_+\right] \colon \LieG_- \to \LieG_-\);
\(
\tau \in \LieG_-^* \otimes \LieG_- \otimes \LieG_-^* \otimes \LieG_-.
\)
We extend \(\tau\) by \(G\)-equivariance to a tensor on \(G/P\), a holomorphic section of 
\(T^* \otimes T \otimes T^* \otimes T\), where \(T=T(G/P)\) is the holomorphic
tangent bundle.

\begin{lemma}
Suppose that \(P \subset G\) is a parabolic subgroup of a complex semisimple Lie group \(G\) in adjoint form.
Split \(P\) into the Langlands decomposition \(P=G_0 G_+\) as above.
The adjoint action of \(G_0\) on \(\LieG/\LieP\) is faithful.
\end{lemma}
\begin{proof}
Elements of \(G\) from different simple factors commute, so we can restrict
our attention to the case where \(G\) is a simple Lie group. 
Take an element \(g_0 \in G_0\) acting trivially on \(\LieG/\LieP=\LieG_-\).
By invariance of the Killing form, \(g_0\) acts trivially on \(\LieG_+\) too, and so  acts trivially on the Lie subalgebra \(\LieK\) of \(G\) generated by \(\LieG_- \cup \LieG_+\).
So \(\LieK\) contains the root spaces of the noncompact roots.
Since both \(\LieG_-\) and \(\LieG_+\) are invariant under reflection in the compact roots, so is \(\LieK\).
By the relation 
\(
\left[X_{\alpha^-},X_{\alpha^+}\right] = H_{\alpha^-},
\)
\(\LieK\) contains the coroots of the noncompact roots.
By the relation
\(
\left[H_{\alpha^-},X_{\gamma}\right] = \gamma\pr{H_{\alpha^-}} X_{\gamma},
\)
\(\LieK\) contains the root spaces of all compact roots \(\gamma\) so that \(\gamma\) is not perpendicular to some noncompact root \(\alpha^-\).
Applying the same equation inductively, \(\LieK\) contains the root space of each simple root \(\gamma\) which lies in the same component of the Dynkin diagram as some noncompact root.
Since \(G\) acts effectively on \(X=G/P\), every component of the Dynkin diagram of \(G\) contains a noncompact root.
Therefore \(\LieK=\LieG\).
So \(g_0\) acts trivially on \(\LieG\), i.e. lies in the center of \(G\), which is \(\left\{1\right\}\) by hypothesis.
\end{proof}

\begin{lemma}
If \(\tau\) is the fundamental tensor of a cominiscule variety \(G/P\) then
\(
\tau \in 
\LieG_-^* \otimes \LieG_- \otimes \LieG_0.
\)
As we vary \(X_-\) and \(\xi^-\), the values of \(t\left(X_-,\xi^-\right)\) span \(\LieG_0\).
\end{lemma}
\begin{proof}
It suffices to prove the result under the assumption that \(G/P\) is irreducible, i.e. \(G\) is simple.
Let \(V \subset \LieG_0\) be the span of all vectors in \(\LieG_0\) of the form \(\left[X_-,Y_+\right]\).
We have only to prove that \(V=\LieG_0\). For any \(Z_0 \in \LieG_0\), 
\[
\left[Z_0,\left[X_-,Y_+\right]\right]
=
\left[\left[Z_0,X_-\right],Y_+\right]
+
\left[X_-,\left[Z_0,Y_+\right]\right]
\]
lies in \(V\). 
So \(V\) is an ideal in \(\LieG_0\). 

If we take the Dynkin diagram of \(\LieG\) and cut out the compact simple roots, we get the Dynkin diagram of \(\LieG_0\).
The connected components in the Dynkin diagram of \(\LieG_0\) are the Dynkin diagrams of simple subalgebras of \(\LieG_0\).
The Lie algebra \(\LieG_0\) also contains all of the noncompact coroots of \(\LieG\), which sum to an abelian subalgebra.
Moreover, \(\LieG_0\) is the sum of this abelian subalgebra with those simple subalgebras.
Take a Chevalley basis. 
For each root \(\alpha^+\), 
\(
H_{\alpha^+} 
= 
\left[X_{\alpha^+},X_{\alpha^-}\right] \in V.
\)
So \(V\) is an ideal in \(\LieG_0\) containing the Cartan subalgebra, and therefore containing the abelian summand of \(\LieG_0\) and intersecting all of the simple summands.
Since \(V\) intersects each simple summand of \(\LieG_0\), and is an ideal, it must contain each simple summand, and contains the abelian summand, so \(V=\LieG_0\).
\end{proof}

Suppose that \(G/P\) is a cominiscule variety.
Let \(\Gamma\) be the group of automorphisms of the Dynkin diagram of \(G/P\), i.e. automorphisms of the Dynkin diagram of \(G\) which leave each component invariant and which fix every noncompact root.
With a Cartan subalgebra of \(G\) chosen, we can canonically realize \(\Gamma\) as a subgroup
of \(G\) leaving the Cartan subalgebra invariant; see Fulton and Harris \cite{Fulton/Harris:1991} p. 498.
To be specific, we just take a Chevalley basis and move around the labels on the simple roots according to \(\Gamma\).
Since \(\Gamma\) acts trivially on the noncompact root, it acts as isomorphisms of the compact roots, and as isomorphisms of the Lie groups \(G,G_-,G_0\) and \(G_+\).

\begin{example}\label{example:Grassmannian}
Among the Dynkin diagrams of Grassmannians, only those of the form \(\Gr{p}{\C{2p}}\) have a symmetry in their Dynkin diagram, interchanging the ends, fixing the noncompact root at the middle.
The Dynkin diagram of the maximal semisimple subgroup of \(G_0\) is precisely the Dynkin diagram of \(G/P\) with the noncompact root removed. 
So in this case, \(G_0\) has disconnected Dynkin diagram, and so \(\LieG_0\) is a direct sum \(\C{} \oplus \LieSL{p,\C{}} \oplus \LieSL{p,\C{}}\).
The group \(\Gamma=\Z{}_2\) interchanges these two subalgebras, while adjoint \(G_0\)-action leaves them each invariant.
The Lie algebra \(\LieG_0\) imbeds as a subalgebra of \(\GL{\LieG/\LieP} = \GL{p^2,\C{}}\).
Therefore the projection operator
\(
\alpha \in \LieGL{p^2,\C{}}^* \otimes \LieGL{p^2,\C{}}
\)
which projects onto the first of these \(\LieSL{p,\C{}}\) subalgebras is invariant under \(G_0\), but not invariant under \(\Gamma\).
To be explicit, write each linear map \(s \colon \LieGL{p,\C{}} \to \LieGL{p,\C{}}\) as \(s(A)=s^{i\ell}_{jk} A^k_{\ell}\), and then
\(
\alpha(s)=t
\)
where
\(
t^{i\ell}_{jk}
=
\frac{1}{n}
\pr{s^{ip}_{pk} - s^{pq}_{pq}\delta^i_k}\delta^{\ell}_j.
\)
\end{example}

\begin{example}\label{example:Hyperquadric}
Quadric hypersurfaces \(Q^{2n-2} \subset \Proj{2n-1}\) have an automorphism from their Dynkin diagram, interchanging the upper and lower roots on the right hand side of the diagram.
The groups are
\[
G=\PO{2n,\C{}}, G_0=\C{\times}\SO{2n-2,\C{}}, \LieG_-=\LieG/\LieP=\C{2n-2}.
\]
The group \(G_0\) acts in the obvious representation on \(\LieG_-\).
This automorphism swaps the two representations in the splitting
\[
\Lm{n-1}{\LieG/\LieP}^* = \Lm{n-1}{\C{2n-2}}^* = \Lambda^+ \oplus \Lambda^-
\]
into self-dual and anti-self-dual \(n\)-forms.
So the group \(\Gamma\) of automorphisms of the Dynkin diagram will not preserve the \(\Lambda^+\) component. 
Let \(\alpha \in \Lm{n-1}{\C{2n-2}} \otimes \Lm{n-1}{\C{2n-2}}^*\) be the projection to \(\Lambda^+\). 
So \(\alpha\) is invariant under \(G_0\) but not under \(\Gamma\). 
Since \(\Lm{n-1}{\LieG_-} \subset \bigotimes^{n-1} \LieG_-\) we think of \(\alpha\) as an element of 
\[
\left(\bigotimes^n \LieG_-\right) \otimes \left(\bigotimes^n \LieG_-^*\right)
=
\bigotimes^n \left( \LieG_- \otimes \LieG_-^* \right).
\]
\end{example}

Suppose that \(G/P\) is an irreducible cominiscule variety. 
If \(G=A_{2p-1}=\PSL{2p,\C{}}\) and \(G/P=\Gr{p}{2p}\), then let \(\alpha\) be the tensor defined in example~\vref{example:Grassmannian}.
If instead \(G=D_n=\SO{2n,\C{}}\) so that \(G/P=Q^{2n-2} \subset \Proj{2n-1}\), then let \(\alpha\) be the tensor defined in example~\vref{example:Hyperquadric}.
Otherwise let \(\alpha=0\).

Suppose that \(G/P\) is a cominiscule variety, factoring into
irreducibles as
\[
G/P = \prod_i G^i/P^i.
\]
Define the \emph{barnacle tensor} to be the tensor that is the formal sum of the various tensors \(\alpha\) on each factor.

\begin{lemma}\label{lemma:IdentityComponent}
Suppose that \(G/P\) is a cominiscule variety.
The kernel of the obvious morphism \(\rho \colon P \to \GL{\LieG/\LieP}\) is \(G_+\) and the image is \(G_0\). 
The group \(G_0\) is precisely the set of linear transformations of \(\LieG/\LieP\) which stabilize the fundamental tensor and the barnacle tensor and preserve the splitting 
\[
\LieG/\LieP = \bigoplus_i
\LieG^i/\LieP^i.
\]
\end{lemma}
Note that the splitting can be defined by a collection of complex linear projection operators.
\begin{proof}
Since \(G\) and \(G_0\) preserve the splitting, we can safely assume that \(G\) is a simple Lie group. 
Let \(G_0'\) be the group of of linear transformations of \(\LieG/\LieP\) which stabilize the fundamental tensor and the barnacle tensor and preserve the splitting.
Since \(G_0'\) is by definition a linear algebraic group, it has finitely many components.
The group \(G_+\) lies in the kernel of \(P \to G_0'\).

Suppose that \(r \in G_0'\).
The image of the fundamental tensor is the Lie algebra \(\LieG_0 \subset \LieGL{\LieG/\LieP}\).
Therefore \(\Ad_r \LieG_0 = \LieG_0\) inside \(\LieGL{\LieG/\LieP}\). 
Since \(G_0\) is connected, \(\Ad_r G_0=G_0\) inside \(\GL{\LieG/\LieP}\),
i.e. \(G_0 \subset G_0'\) is a normal subgroup.

Define a Lie algebra
\(
\LieG'' =  \LieG_- \oplus \LieG_0 \oplus \LieG_-^*,
\)
by using the usual brackets on \(\LieG_0\), and by making  \(\LieG_-\) and \(\LieG_-^*\) abelian (so \(0\) brackets), and by making the brackets of \(\LieG_0\) on \(\LieG_-\) and on \(\LieG_-^*\) be the obvious action as linear transformations. 
Then finally define \(\left[X_-,\xi^-\right]=t\left(X_-,\xi^-\right)\) for \(X_- \in \LieG_-\) and \(\xi^- \in \LieG_-^*\). 
This bracket is \(G_0'\)-invariant, and \(\LieG''=\LieG\) are isomorphic Lie algebras. 
Moreover, \(\LieP \subset \LieG\) is an \(G_0'\)-invariant subalgebra. 

So \(G_0'\) acts as automorphisms of the connected Lie group \(G\) preserving the connected subgroups \(G_0,G_+,G_-,P\), say by a morphism \(\sigma \colon G_0' \to \Aut{G}\).
This morphism \(\sigma\) fixes \(G_0\) to give the usual morphism \(\Ad \colon G_0 \to G \subset \Aut{G}\).
The automorphism group \(\Aut G\) is a finite extension of \(\Ad G\) (where \(\Ad G\) acts on \(G\) by inner automorphisms), since \(G\) is semisimple (see Fulton and Harris \cite{Fulton/Harris:1991} p. 498).
Since \(G\) acts effectively on \(X=G/P\), \(G = \Ad G\), so \(\Aut G\) is a finite extension of \(G\).
The morphism \(\sigma\) takes \(\LieG_0 \subset \LieG'_0 \to \LieG_0\), so takes \(G_0 \subset G'_0 \to G_0 \subset G\), since \(G_0\) is connected, so \(G_0\) is embedded as a subgroup of \(\GL{\LieG/\LieP}\).

The inclusion map \(G_0' \to \Aut{G}\) gives an inclusion \(\LieG_0' \to \LieG\) on Lie algebras, extending the inclusion \(\LieG_0 \to \LieG\). 
The Cartan subalgebra of \(\LieG\) is a subalgebra of \(\LieG_0'\), so \(\LieG_0'\) is a sum
of root spaces. 
The noncompact root spaces don't preserve the subspaces \(\LieG_-, \LieG_+ \subset \LieG\), so none of these root spaces lie in \(\LieG_0'\).
Since \(\LieG_0\) is an ideal of \(\LieG_0'\), the compact root spaces lie inside \(\LieG_0'\), so \(\LieG_0'=\LieG_0\).
In particular \(G_0\) is the identity component of \(G_0'\). 

Every automorphism of a complex semisimple Lie group \(G\) factors uniquely as \(kg\) where \(g \in G\) and \(k\) is an automorphism of the Dynkin diagram; see Fulton and Harris \cite{Fulton/Harris:1991} p. 498.
So each element of \(G_0'\) must factor as \(kg\) where \(k\) lies in the automorphism group of the Dynkin diagram of \(G\), and \(g \in G\). 
This same element must act as an automorphism of \(G_0\). 

The group of inner automorphisms of a complex affine group \(G_0\) acts transitively on the Cartan subgroups; Borel \cite{Borel:1991} p. 156. 
Therefore if we pick any element \(g_0' \in G_0'\), we can arrange after multiplication with
some \(g_0 \in G_0\) that \(g_0'\) acts on \(G_0\) preserving its Cartan subgroup, which is also the Cartan subgroup of \(G\).
The element \(g_0'\) therefore acts as an automorphism of the Dynkin diagram of \(G\),
leaving invariant the roots whose root spaces belong to \(\LieG_0\), i.e. as an automorphism
of the Dynkin diagram of \(G/P\). 
Examining the Dynkin diagrams, there are only two cases of irreducible cominiscule varieties whose Dynkin diagrams have nontrivial automorphism groups: for \(A_{2p-1}\), the Grassmannian \(\Gr{p}{\C{2p}}\), and \(D_n\), the quadric \(Q^{2n-2}\). 
In each case the automorphism has order 2, and does not preserve the barnacle tensor. Therefore \(G_0'=G_0\).
\end{proof}

The group of linear transformations of \(\LieG/\LieP\) preserving the fundamental tensor and splitting is already a finite extension of \(G_0\). 
The barnacle tensor is defined here in order to get rid of the order two automorphisms of the Grassmannians \(\Gr{p}{\C{2p}}\) and the quadric hypersurfaces.

\subsection{The curvature boundary operator}

Suppose that \(G/P\) is a cominiscule variety.
Following Calderbank and Diemer \cite{Calderbank/Diemer:2001}, given any representation \(\rho \colon P \to \GL{W}\), we define an associated space of \emph{chains}
\(
C_k\left(W\right) = \Lm{k}{\LieG_-}^* \otimes W,
\)
and a linear map which we call the \emph{boundary operator}
\(
\delta \colon C_k\left(W\right) \to C_{k-1}\left(W\right),
\)
by
\[
\delta\left(\beta \otimes w\right)
=
\sum_{\alpha_-}
\left(X_{\alpha_-} \hook \beta\right)
\otimes \rho\left(X_{\alpha_+}\right)w,
\]
for \(\beta \in \Lm{k}{\LieG_-}^*\) and \(w \in W\).

\subsection{The torsion operator}

Define the \emph{torsion operator}
\[
\torop \colon
\LieG_-^* \otimes 
\LieG_0
\to \Lm{2}{\LieG_-}^* \otimes \LieG_-,
\]
as follows: for any \(a \in \LieG_-^* \otimes \LieGL{\LieG_-}\),
and any \(X,Y \in \LieG_-\), let
\[
\torop(a)(X \wedge Y)=\left[a(X),Y\right]-\left[a(Y),X\right].
\]

\begin{lemma}\label{lemma:GplusActsFaithfullyonGminus}
Suppose that \(G/P\) is a cominiscule variety, with grading
\(
\LieG = \LieG_- \oplus \LieG_0 \oplus \LieG_+.
\)
The morphism 
\(
\ad \colon \LieG_+ \to \LieG_-^* \otimes \LieG_0,
X_+ \mapsto \ad\left(X_+\right)
\)
is injective.
\end{lemma}
\begin{proof}
Suppose that \(K_+\) is the kernel of this morphism.
It follows from the Leibnitz identity that \(K_+\) is a \(\LieG_0\)-module. 
Therefore \(K_+\) is a sum of root spaces of \(\LieG\).
So it is enough to prove that for any root \(\alpha^+\), \(X_{\alpha^+}\) is not in \(K_+\), which follows from
\(
\left[X_{\alpha^+}, X_{\alpha^{-}}\right]
=
H_{\alpha^+} \ne 0.
\)
\end{proof}

\begin{lemma}
The map \(\ad \colon \LieG_+ \to \LieG_-^* \otimes \LieG_0\) has image inside \(\ker \torop\).
\end{lemma}
\begin{proof}
Pick a root \(\gamma\).
Pick \(X_{\gamma}\) and \(X_{-\gamma}\) bases of the root spaces \(\LieG_{\gamma}\) and \(\LieG_{-\gamma}\).
The kernel of \(\ad\) must be a sum of root spaces, by equivariance under the Cartan subalgebra.
By the Jacobi identity, for \(X_{+} \in \LieG_+\) and \(Y_{-}, Z_{-} \in \LieG_-\),
\begin{align*}
\left[
	\ad_{X_{+}}\left(Y_{-}\right),Z_{-}
\right]
&=
-\left[
	\left[
		Z_{-},X_{+}
	\right]
	,Y_{-}
\right]
-
\left[
	\left[
		Y_{-},Z_{-}
	\right]
	,X_{+}
\right]
\\
&=
\left[
	\ad_{X_{+}}
	\left(
		Z_{-}
	\right)
	,Y_{-}
\right].
\end{align*}
\end{proof}

\subsection{The torsion operator in a Chevalley basis}

Let
\[
H = 
\sum_{\beta^{-}} H_{\beta^{-}}.
\]
Then for any \(a \in \LieG_-^* \otimes \LieG_0\),
define \(\tortr{a} \in \LieG_+\), the \emph{trace}
of \(a\),
by
\[
\tortr{a}=
\sum_{\beta^-, \delta^-}
\frac
	{
	\left<
		\left[
			a
				\left(
					X_{\beta^{-}}
				\right)
				,
					X_{\beta^{+}}
		\right]
		,X_{\delta^{-}}
	\right>
	}
{
	\delta^{-}(H) 
	\left<
		X_{\delta^{-}},
		X_{\delta^{+}}
	\right>
}
X_{\delta^{+}}.
\]

It is easy to check that for any \(X_{+} \in \LieG_+\), if \(a = \ad_{X_+}\) then  \(\tortr{a}=X_+\).

\begin{lemma}
If \(a \in \LieG_-^* \otimes \LieG_0\) and \(Da=0\) then
\(
a = \ad_{X_+} + b,
\)
for a unique \(X_+ \in \LieG_+\) (given by \(X_+ = \tortr{a}\)) and a unique \(b \in \LieG_-^* \otimes \LieG_0\) for which \(Db=0\) and \(\tortr{b}=0\).
\end{lemma}
\begin{proof}
Let \(X_+ = \tortr{a}\) and let \(b=a-\ad_{X_+}\).
\end{proof}

\subsection{Tensor invariants in cominiscule geometries}

Suppose that \(P \to E \to M\) is a cominiscule geometry, with Cartan connection \(\omega\), over a vector subbundle \(V_M \subset TM\).
The identification 
\[
\begin{tikzcd}
0 \arrow{r} & \text{vertical}_e \arrow{r} \arrow{d}{\omega_e} & V_e E \arrow{r} \arrow{d}{\omega_e} & V_m M \arrow{r} \arrow{d} & 0 \\
0 \arrow{r} & \LieP \arrow{r} & \LieG \arrow{r}  & \LieG/\LieP \arrow{r} & 0 
\end{tikzcd}
\]
identifies the fundamental tensor \(\tau\) of \(G/P\) with a tensor on \(V_m M\), which we will also call \(\tau\).
Because \(\tau\) is a \(P\)-invariant tensor on \(\LieG/\LieP\), this identification is independent of the point \(e \in E\), determining a holomorphic section \(\tau\) of \(V^* \otimes V \otimes V^* \otimes V\), where \(V=V_M\).

If \(\LieG = \bigoplus_i \LieG^i\), let \(p_i \colon \LieG/\LieP \to \LieG^i/\LieP^i\) be the obvious projection. 
The same identification yields a splitting \(V = \bigoplus_i V_i\) corresponding to the splitting \(\LieG/\LieP = \bigoplus_i \LieG^i/\LieP^i\), and a holomorphic projection tensor \(\varpi^i \in V^* \otimes V_i\) identified with \(p_i\). 
It also yields a barnacle tensor \(\alpha\), a section of \(V^* \otimes V_i\) on \(M\).
A cominiscule geometry on a foliation is \emph{normal} if its curvature is \(\delta\)-closed.

\begin{lemma}
Suppose that \(E \to M\) is a cominiscule geometry on a vector bundle \(V_M \subset TM\), with model \(G/P\). 
The underlying first order structure of \(E \to M\) is the \(G_0\)-structure \(E/G_+ \to \frm{V_M}\). 
Its image in \(\frm{V_M}\) consists precisely in the linear isomorphisms 
\(
u \colon V_m M \to \LieG/\LieP
\)
preserving the canonical splitting tensors, fundamental tensor and barnacle tensor.
In particular, two cominiscule geometries have the same underlying first order structure just precisely when they have the same canonical splitting, fundamental tensor and barnacle tensor.
\end{lemma}
\begin{proof}
See lemma~\ref{lemma:IdentityComponent}.
\end{proof}

\section{Cominiscule geometries and underlying first order structures}%
\label{section:pairs}

Suppose that \(\pi \colon E \to M\) is a Cartan geometry over a vector bundle \(V_M\), with model \(X=G/H\). 
Write the composition \(E \to E/H_1 \to \frm{V_M}\) as \(\pi^{(1)}\).
Suppose that \(E \to M\) and \(E' \to M\) are two holomorphic Cartan geometries on the same vector bundle \(V_M \subset TM\), with the same model \(X=G/H\), and with the same underlying first order structure.
Let \(E \times_{\frm{V_M}} E'\) be the set of all tuples
\(
\left(m,u,e,e'\right),
\)
so that \(m \in M\) and \(u \colon V_m M \to \LieG/\LieH\) is a linear isomorphism and \(e \in E\) and \(e' \in E'\) and
\[
u=\pi^{(1)}(e) = \left(\pi'\right)^{(1)}\left(e'\right).
\]
The bundle \(E \times_{\frm{V_M}} E' \to M\) is a holomorphic principal right \(H \times_{H/H_1} H\)-bundle. 
The obvious maps \(E \times_{\frm{V_M}} E' \to E\) and \(E \times_{\frm{V_M}} E' \to E'\) are each holomorphic principal right \(H_1\)-bundles.

\begin{lemma}\label{lemma:SolderingMatch}
Suppose that \(E \to M\) and \(E' \to M\) are two Cartan geometries over the same vector bundle \(V_M \subset TM\) with the same model \(X=G/H\), and the same underlying first order structure. 
If we write the pullback of any form by the same symbol, then on \(E \times_{\frm{V_M}} E'\), 
\(\sigma = \omega' + \LieH = \omega+\LieH\).
\end{lemma}
\begin{proof}
By symmetry it is enough to prove that \(\omega+\LieH = \sigma\), and it is enough to prove this on \(E\), which we see by unwinding definitions. 
\end{proof}

\begin{proposition}\label{proposition:DefineBandXPlus}
Suppose that \(E \to M\) and \(E' \to M\) are two cominiscule geometries with same model \(G/P\), on the same vector bundle \(V_M \subset TM\) and with the same fundamental, splitting and barnacle tensors. 
On \(E \times_{\frm{V_M}} E'\), there are unique holomorphic functions
\(
s \colon E \times_{\frm{V_M}} E' \to \LieG_-^* \otimes \LieG_0,
\)
and
\(
X_+ \colon E \times_{\frm{V_M}} E' \to \LieG_+,
\)
so that 
\begin{enumerate}
\item \(\tortr{s}=0\),
\item \(\omega'_0 = \omega_0 + \left(s +  \ad\left(X_{+}\right)\right) \, \sigma\),
\item \(s\) and \(X_+\) are equivariant under the \(P \times_{G_0} P\)-action.
\end{enumerate}
\end{proposition}
\begin{proof}
From lemma~\vref{lemma:SolderingMatch}, \(\omega_- = \omega_-' = \sigma\). 
Taking exterior derivative of some local extension of \(\omega\) and \(\omega'\) to equivariant differential forms:
\[
0=
\frac{1}{2}\left[\omega_0-\omega_0',\omega_-\right]
+
\frac{1}{2}\left[\omega_-,\omega_0-\omega_0'\right]
-
\frac{1}{2}\left(K_- - K_-'\right)\sigma \wedge \sigma,
\]
where \(K \sigma \wedge \sigma\) and \(K' \sigma \wedge \sigma\) are the curvatures, and \(K_-\) and \(K'_-\) are the components of \(K\) and \(K'\) valued in \(\LieG_-\).
Pick any vector \(v \in V_{E \times_{\frm{V_M}} E'}\) for which \(v \hook \sigma=0\).
Let \(A_0 = v \hook \left(\omega_0 - \omega_0'\right)\).
Then plug in \(v\) to the above to find
\[
0 = \frac{1}{2}\left[A_0,\omega_-\right]
+
\frac{1}{2}\left[\omega_-,A_0\right].
\]
In other words, if \(B_- \in \LieG_-\), then
\(
0 = \left[A_0,B_-\right].
\)
By lemma~\vref{lemma:IdentityComponent}, \(\LieG_0\) is a subalgebra of \(\LieGL{\LieG_-}\),
ie. this forces \(A_0=0\). 
Therefore on the kernel of \(\omega_-=\sigma\), we have \(\omega_0=\omega_0'\), so
\(
\omega_0' = \omega_0 + \left(s + \ad_{X_+}\right)\sigma,
\)
where \(s \in \LieG_-^* \otimes \LieG_0\) and \(\tortr{s} = 0\).
\end{proof}

\begin{lemma}\label{lemma:expPlus}
Suppose that \(E \to M\) is a cominiscule geometry with model \(G/P\) and Cartan connection \(\omega\).
If \(X_+ \in \LieG_+\) and we let \(g_+ = e^{X_+}\),
then
\[
r_{g_+}^* 
\begin{pmatrix}
\omega_- \\
\omega_0 \\
\omega_+
\end{pmatrix}
= 
\begin{pmatrix}
I & 0 & 0 \\
-\ad\left(X_+\right) & I & 0 \\
\frac{1}{2} \ad\left(X_+\right)^2 & -\ad\left(X_+\right) & 0
\end{pmatrix}
\begin{pmatrix}
\omega_- \\
\omega_0 \\
\omega_+
\end{pmatrix}
\]
\end{lemma}
\begin{proof}
From the definition of a Cartan connection, \(r_{g_+}^* \omega = \Ad\pr{g_+}^{-1} \omega\). The result then follows from equation~\vref{equation:AdexpX}.
\end{proof}

\begin{theorem}\label{theorem:Difference}
Any two cominiscule geometries on the same vector bundle, with the same model and underlying first order structure have canonically isomorphic total space.
To be more specific, suppose that \(E \to M\) and \(E' \to M\) are two cominiscule geometries with the same model \(G/P\) and on the same holomorphic vector bundle \(V_M \subset TM\).
Suppose that \(E\) and \(E'\) have the same fundamental tensor, barnacle tensor and splitting. Then there exists a unique bundle isomorphism \(E \cong E'\) so that \(\omega' = \omega + S \, \sigma\) for a unique holomorphic \(P\)-equivariant function \(S \colon E \to \pr{\LieG/\LieP}^* \otimes \LieP\) with \(\tortr S = 0\). 
The two cominiscule geometries are isomorphic, via an isomorphism that is the identity on \(M\), if and only if \(S=0\).

Conversely, suppose we have a cominiscule geometry \(E \to M\) with model \(G/P\) and Cartan connection \(\omega\). 
Every holomorphic section \(S\) of \(T^*M \otimes \ad(E)\) satisfying \(\tortr S=0\) determines a Cartan connection \(\omega'=\omega+S \, \sigma\), which has the same fundamental, splitting and barnacle tensors.
\end{theorem}
Call \(S\) the \emph{obstruction} of the pair \(E,E'\).
Let \(\Sch_0 \subset \pr{\LieG/\LieP}^* \otimes \LieP\) be the \(P\)-submodule consisting of all \(s \in \pr{\LieG/\LieP}^* \otimes \LieP\) so that \(\tortr s = 0\).
For any holomorphic \(G_0\)-structure, let \(\Sch = E \times_P \Sch_0\) so that \(\Sch \to M\) is a holomorphic vector bundle.
Similarly, if we have a local biholomorphism \(f \colon M \to M'\) and \(G/P\)-geometries \(E \to M\) and \(E' \to M'\), and \(f\) locally identifies their underlying first order structures, define the obstruction of \(f\) to mean the obstruction of the pair \(\pr{E,f^*E'}\).
\begin{proof}
Under the \(P \times_{G_0} P\)-action on \(E \times_{\frm{V_M}} E\), the action on \(\omega\) and \(\omega'\) is obvious. 
Pick \(\left(g_0,g_0\right) \in G_0 \times_{G_0} G_0=G_0\).
Let \(\rho_- \colon G_0 \to \GL{g_-}\) and \(\rho_+ \colon G_0 \to \GL{g_+}\) be the obvious representations.
Then check that 
\[
r_{\left(g_0,g_0\right)}^* S_0 = \Ad \left(g_0\right)^{-1} \left( S_0 \circ \rho_-\left(g_0\right)\right).
\]
Similarly,
\(
r_{\left(g_0,g_0\right)}^* X_+ = \rho_+\left(g_0\right)^{-1} X_+.
\)

Pick \(Y_+, Y_+' \in \LieG_+\). Then let
\(
g_+ = e^{Y_+}, g_+' = e^{Y_+'}.
\)
By lemma~\vref{lemma:expPlus}, on \(E \times_{\frm{V_M}} E'\),
\begin{align*}
r_{\left(g_+,g_+'\right)}^*
\left(\omega_0' - \omega_0\right)
&=
r_{\left(g_+,g_+'\right)}^*
\left(
  \left(
    S + \ad
           \left(
              X_+
           \right) 
   \right) 
   \, \omega_-
\right)
\\
&=
\left(
  S + \ad
         \left(
            X_+-Y_+' + Y_+
         \right)
\right) 
\, \omega_- \\
&=
\left(
  r_{\left(g_+,g_+'\right)}^* S
  +
  \ad
    \left(
      r_{\left(g_+,g_+'\right)}^* X_+
    \right)
\right) 
\, \omega_-.
\end{align*}
Let \(F\) be the set of points of \(E \times_{\frm{V_M}} E'\) at which the function \(X_+\) (defined in proposition~\vref{proposition:DefineBandXPlus}) vanishes. 
The set \(F\) is a principal right \(P\)-subbundle.
Indeed its tangent space is a maximal integral manifold of the Pfaffian system \(\tortr{\left(\omega_0'-\omega_0\right)}=0\).
Therefore both maps \(F \to E\) and \(F \to E'\) are bundle isomorphisms.
We henceforth identify \(E\) with \(E'\) by \(E \to F \to E'\).

The function \(S_0 \colon E \to \LieG_-^* \otimes \LieG_0\) is \(P\)-equivariant, and therefore \(G_+\)-invariant.
Consequently if \(v\) is any tangent vector to \(E\) for which \(0=v \hook \omega_-=v \hook \omega_0\) then \(v \hook dS_0=0\) and \(S_0\) is \(G_0\)-equivariant.
If \(v\) is any tangent vector to \(E\) for which \(0 = v \hook \omega_-\) and \(v \hook \omega_0 = A_0\), then
\(
v \hook dS_0 = - \ad\left(A_0\right)S_0 + S_0 \circ \rho\left(A_0\right).
\)
Therefore
\begin{equation}\label{equation:da}
dS_0 = - \ad\left(\omega_0\right)S_0 + S_0 \circ \rho\left(\omega_0\right) + \nabla S_0 \, \omega_-,
\end{equation}
where \(\nabla S_0 \colon E \to \LieG_-^* \otimes \LieG_-^* \otimes \LieG_0\) is a holomorphic \(P\)-equivariant map.
The equation \(\tortr{S}=0\) forces \(\tortr{dS}=0\), so forces \(\tortr{\nabla S}=0\).

Take exterior derivative of both sides of the equation
\(
\omega_0' = \omega_0 + S \, \omega_-
\)
to find
\begin{align*}
0 
=&
-\frac{1}{2}\left[S \, \omega_-, \omega_0\right]
-\frac{1}{2}\left[\omega_0, S \, \omega_-\right]
-\frac{1}{2}\left[S \, \omega_-, S \, \omega_-\right]
-\frac{1}{2}\left[\omega_-, \omega_+' - \omega_+\right]
-\frac{1}{2}\left[\omega_+' - \omega_+, \omega_-\right]
\\
&+\frac{1}{2}\left(K_0' - K_0\right) \omega_- \wedge \omega_-
- dS \wedge \omega_-
+S
\left(
\frac{1}{2} \left[\omega_-,\omega_0\right]
+
\frac{1}{2} \left[\omega_0,\omega_-\right]
-
\frac{1}{2} \left[\omega_-,\omega_0\right]
\right).
\end{align*}
Plugging in what we know about \(dS\) from equation~\vref{equation:da},
\begin{align*}
0 
=&
-\frac{1}{2}\left[S \, \omega_-, S \, \omega_-\right]
-\frac{1}{2}\left[\omega_-, \omega_+' - \omega_+\right]
-\frac{1}{2}\left[\omega_+' - \omega_+, \omega_-\right]
\\
&+\frac{1}{2}\left(K_0' - K_0 - S K_-\right) \omega_- \wedge \omega_-
-
\left(\nabla S \omega_-\right) \wedge \omega_-.
\end{align*}

Pick a vector \(v\) tangent to \(E\) for which \(0=v \hook \omega_-\).
Suppose that \(v \hook \omega_+ = A_+\) and \(v \hook \omega_+' = A_+'\).
Plug in to find
\[
0 = \frac{1}{2}\left[\omega_-,A_+' - A_+\right]
-\frac{1}{2}\left[A_+' - A_+,\omega_-\right].
\]
If we then pick any \(B_- \in \LieG_-\) and take a vector \(w\) tangent to \(E\) with
\(w \hook \omega_- = B_-\), we find
\(
0 = \left[B_-, A_+'-A_+\right],
\)
i.e. \(\ad\left(A_+' - A_+\right)\) acts trivially as a linear map in \(\LieG_-^* \otimes \LieG_0\).
By lemma~\vref{lemma:GplusActsFaithfullyonGminus},
\(A_+'=A_+\). Therefore \(\omega_+' = \omega_+\) modulo \(\omega_-\).
So \(\omega_+'=\omega_+ + b \, \omega_-\)
for some unique holomorphic function
\(
S_+ \colon E \to \LieG_-^* \otimes \LieG_+.
\)
Set \(S=S_0+S_+\); the required morphism exists. 

Suppose that there are two isomorphisms
\begin{align*}
f & \colon E \to E', f^* \omega' = \omega + S \sigma, \\
\tilde{f} & \colon E \to E', \tilde{f}^* \omega' = \omega + \tilde{S} \sigma.
\end{align*}
Since \(f^* \omega'_- = \omega_- = \sigma\), it follows that \(E\) and \(E'\) have the same underlying first order structure.
Construct the set \(F\) as above: the set of points of \(E \times_{\frm{V_M}} E'\)
at which the function \(X_+\) (defined in proposition~\vref{proposition:DefineBandXPlus})
vanishes. 
It follows that the graphs of \(f\) and \(\tilde{f}\) both lie inside \(F\). 
Being principal \(P\)-bundles over \(M\), they must therefore be equal to \(F\) and to one another.
\end{proof}

A \emph{projective factor} of a cominiscule variety
\[
G/P = \left(G^1/P^1\right) \times \left(G^2/P^2\right) \times \dots \times \left(G^s/P^s\right).
\]
is a factor \(G^i/P^i = \Proj{n}\) with \(G^i=\PSL{n+1,\C{}}\).
A \emph{projective line factor} is a projective factor \(G^i/P^i=\Proj{1}\).

\begin{lemma}[\v{C}ap and Slovak \cite{Cap/Slovak:2009} p. 277 theorem 3.1.16, Goncharov \cite{Goncharov:1987}]%
\label{lemma:Cap.Schichl.Goncharov}
Suppose that \(\pr{G,X}\) is a cominiscule variety.
If a complex manifold admits a holomorphic \(\pr{G,X}\)-geometry, then the underlying holomorphic first order structure is a holomorphic \(G_0\)-structure.
Conversely, if a complex manifold admits a holomorphic \(G_0\)-structure, then this \(G_0\)-structure is locally the underlying first order structure of a normal holomorphic \(\pr{G,X}\)-geometry.
If \(\pr{G,X}\) contains no projective factors, then the normal holomorphic \(\pr{G,X}\)-geometry is unique.
\end{lemma}
The proofs of \cite{Cap/Slovak:2009} and \cite{Goncharov:1987} are in the real category, but they work identically in the holomorphic category.
The algorithm of \v{C}ap and Slovak assumes a fundamental tensor and splitting, but doesn't assume a barnacle tensor. 
However, the Cartan geometry is defined on a covering space corresponding precisely
to the possible choices of barnacle tensor. 
It is trivial to modify the proof of {\v{C}}ap and Slovak to include a barnacle tensor.

\begin{lemma}\label{lemma:normal.space}
Suppose that \(G/P\) is a cominiscule variety, \(M\) is a complex manifold and \(V_M \subset TM\) is a holomorphic subbundle.
If a \(V_M\) admits a holomorphic \(G/P\)-geometry, then the underlying holomorphic first order structure is a holomorphic \(G_0\)-structure.
Conversely, if \(V_M\) admits a holomorphic \(G_0\)-structure, then every \(G_0\)-structure on \(V_M\) is locally the underlying first order structure of a holomorphic \(G/P\)-geometry.
If \(V_M\) is bracket closed and admits a holomorphic \(G_0\)-structure, then every \(G_0\)-structure on \(V_M\) is locally the underlying first order structure of a normal holomorphic \(G/P\)-geometry.
If \(G/P\) contains no projective factors, then then normal holomorphic \(G/P\)-geometry is unique.
For any two local choices of holomorphic \(G/P\)-geometries with the given \(G_0\)-structure, wherever both are defined, there is a unique holomorphic bundle isomorphism between their principal bundles so that they differ by an obstruction.
\end{lemma}
\begin{proof}
If \(V_M=TM\) then this is just lemma~\vref{lemma:Cap.Schichl.Goncharov}.
Next, suppose that \(M\) is a product \(M=X \times Y\) of complex manifolds, and that \(V_M=TF\) is the tangent bundle of the foliation \(F\) whose leaves are \(\left\{x\right\} \times Y\) for \(x \in X\).
Then we carry out the same algorithm as \v{C}ap and Schichl \cite{Cap/Schichl:2000}, but carrying around a point \(x \in X\) with us as we do.
Moreover the algorithm specifies a choice of projective connection (for any projective factors) using a choice of some additional local data.
An arbitrary foliation is locally a product, and our result is local, so the proof for \(V_M\) the tangent bundle of any holomorphic foliation is immediate.

Next suppose that we have an arbitrary holomorphic subbundle \(V_M \subset TM\).
Since our problem is local, we can suppose that \(M=M_1 \times M_2\), splitting the tangent bundle into \(T=T^1 \oplus T^2\) with \(T^j_{\pr{m_1,m_2}} = T_{m_j} M_j\), so that \(V=V_M\) is complementary to \(T^1\) at every point.
Make the obvious linear identification \(V_m \to T^2_m\).
Use this to push the first order structure on \(V\) to a first order structure on \(T^2\).
Use the above process to locally generate a \(G/P\)-geometry \(E \to M\) with Cartan connection \(\omega'\).
Locally trivialize \(E = M_1 \times M_2 \times P\).
If \(v \in V_E\), write \(v=\left(v_1,v_2,A\right) \in T^1 \oplus T^2 \oplus TP\) and define \(\omega\) by \(\omega(v)=\omega'(0,v_2,A)\).
\end{proof}

\begin{corollary}\label{corollary:glue}
Suppose that \(G/P\) is a cominiscule variety.
Suppose that \(V_M \subset TM\) is a holomorphic vector bundle on a complex manifold \(M\) bearing a holomorphic \(G_0\)-structure.
There is a holomorphic principal \(P\)-bundle \(E \to M\) and a holomorphic map \(E \to FV_M\) with image equal to the \(G_0\)-structure, so that for any local biholomorphism \(X \to M\), every holomorphic \(G/P\)-geometry on \(X\) with the pullback \(G_0\)-structure as its underlying first order structure has bundle canonically holomorphically isomorphic to the pullback of \(E\).
\end{corollary}
\begin{proof}
We construct bundles \(E_{\alpha}\) locally by the previous corollary, and then we (must) glue them together as in theorem~\vref{theorem:Difference}.
\end{proof}
Cover \(M\) in open sets \(U_{\alpha}\), each equipped with a holomorphic \(G/P\)-geometry \(E_{\alpha} \to U_{\alpha}\) for which the \(G_0\)-structure is the underlying first order structure.
Let \(U_{\alpha \beta}=U_{\alpha} \cap U_{\beta}\), and if \(\phi \colon X \to M\) is any bundle, write \(X_{\alpha}\) to mean \(\phi^{-1}U_{\alpha}\) and \(X_{\alpha \beta}\) to mean \(\phi^{-1}U_{\alpha \beta}\).
Over \(U_{\alpha \beta}\), identify \(E_{\alpha}\) with \(E_{\beta}\) as in corollary~\vref{corollary:glue} to make a bundle \(E\), arranging that
\[
\omega^{\beta} = \omega^{\alpha} + S^{\alpha \beta} \sigma
\] 
so that 
\(
S^{\alpha \beta}
\)
is a holomorphic section of \(\Sch_{\alpha \beta}\).
The obstruction \(S=\pr{S^{\alpha \beta}}\) is a 1-cocycle \(S \in \Cohom{1}{M,\Sch}\).
Since each \(\omega^{\alpha}\) is a holomorphic partial connection on \(E_{\alpha} \times_P G\) over \(V_M\), the \(S^{\alpha \beta}\) are differences between holomorphic partial connections on \(E \times_P G\).
Thus the linear embedding \(\Sch_0 \to \pr{\LieG/\LieP} \otimes \LieG \) induces a sheaf morphism \(\Sch \to V_M^* \otimes \ad\pr{E \times_P G}\) which takes the obstruction cocycle \(S\) to the Atiyah class \(a\pr{M,V_M,E\times_P G}\), giving the obstruction to constructing a holomorphic partial connection on \(E \times_P G\) over \(V_M\).

\begin{lemma}\label{lemma:S.obstruction}
Suppose that \((G,X)\) is a cominiscule variety.
A holomorphic \(G_0\)-structure \(B\) on a holomorphic vector bundle \(V_M \subset TM\) is the underlying first order structure of a holomorphic \((G,X)\)-geometry if and only if the obstruction of that structure vanishes in \(\Cohom{1}{M,\Sch}\). 
If the obstruction vanishes, then the space of all holomorphic \((G,X)\)-geometries with \(B\) as underlying first order structure, modulo any isomorphism which is the identity on \(M\), is a complex affine space modelled on the complex vector space \(\Cohom{0}{M,\Sch}\).
\end{lemma}

For example, if \(M\) is a Stein manifold then any holomorphic \(G_0\)-structure on \(M\) is induced by a holomorphic \(G/P\)-geometry, since the cohomology groups in dimension 1 or higher of any coherent sheaf on any Stein manifold all vanish.

If \(G/P\) has no projective factors, then the obstruction of any holomorphic \(G_0\)-structure vanishes, because the uniqueness of the normal holomorphic \((G,X)\)-geometry gives a canonical global choice of a normal \((G,X)\)-geometry, and the obstruction 1-cocycle class in \(H^1\) is precisely the obstruction to such a choice.
Therefore for any rational homogeneous variety \(G/P\), the obstruction of any holomorphic \(G_0\)-structure is precisely identified with the sum of the obstructions of the projective factors: if these vanish, we can then put a product normal holomorphic \(G/P\)-geometry on the manifold.

\subsection{The Schwarzian derivative and projective connections on foliations}

If \(Z=f(z)\) is a local biholomorphism between open sets in the complex plane, the  \emph{Schwarzian derivative} is 
\[
s(f)=\pr{\frac{Z'''(z)}{Z'(z)}-\frac{3}{2}\pr{\frac{Z''(z)}{Z'(z)}}^2} \, dz^2.
\]
On the other hand, if \(Z=f(z)\) is a local biholomorphism between open sets of \(\C{p}\) for some \(p \ge 2\), then let \(J=\det Z'(z)\) and the \emph{Schwarzian derivative} is
\[
s(f)=\pr{\frac{\partial_{ij}Z^{\ell}}{\partial_k Z^{\ell}}
- \frac{\delta^k_j \partial_i J}{p+1} 
- \frac{\delta^k_i \partial_j J}{p+1} 
} \, dz^i \otimes dz^j \otimes \partial_k. 
\]
Gunning \cite{Gunning:1978} pp. 48--50 proves that \(Z=f(z)\) is a projective linear transformation just when \(s(f)\) vanishes.
More generally, if \((Z,W)=f(z,w)\) is a local biholomorphism between open sets of \(\C{p} \times \C{q}\), and \(f\) takes the level sets of \(w\) to the level sets of \(W\), we use the same expression to define \(s(f)\) but viewing \(Z\) as if it were a function of \(z\) only.
Then \(s(f)=0\) just when \((Z,W)=f(z,w)\) has \(Z(z,w)\) a holomorphic projective transformation of \(z\) for each fixed \(w\).

Suppose that \(F\) is a holomorphic nonsingular foliation of a complex manifold \(M\).
A Frobenius coordinate chart on \(M\) is a coordinate chart \((z,w)\) defined on a Stein domain in \(M\) so that each leaf of \(F\) intersects the coordinate chart in a union of level sets of \(w\).
If \((z,w)\) and \((Z,W)\) are any two Frobenius coordinate charts, on the overlap where both are defined we can compute \(s(f)\) where \((Z,W)=f(z,w)\).
Pick Frobenius coordinate charts \(\pr{z_{\alpha},w_{\alpha}}\) on open sets \(U_{\alpha}\) covering \(M\), so that any finite intersection of these open sets is a Stein domain.
Write the overlaps of charts as \(\pr{z_{\beta},w_{\beta}} = f_{\alpha \beta}\pr{z_{\beta},w_{\beta}}\) and let \(s_F^{\alpha \beta} = s\pr{f_{\alpha \beta}}\).
By Gunning's result, there is a flat holomorphic projective connection on \(M\) just when there are Frobenius charts \(\pr{z_{\alpha},w_{\alpha}}\) whose domains cover \(M\) so that \(s^{\alpha \beta}=0\).

We can ask for much less than a flat holomorphic projective connection.
Let \(TF\) be the tangent bundle of \(F\), i.e. the bundle of tangent vectors tangent to the leaves of \(F\).
If we take any tensor products of \(TF\) and \(T^*F\), we write \(\otimes_0\) to mean the subbundle of all tensors that are traceless in all possible ways, corresponding to an irreducible representation of the general linear group.
If the leaves of \(F\) are 1-dimensional, then \(s_F\) determines a 1-cocycle \(s_F \in \Cohom{1}{M,-2 \, TF}\).
On the other hand, if the leaves of \(F\) have dimension \(p \ge 2,\) then 
\(s_F\) determines a 1-cocycle \(s_F \in \Cohom{1}{M,TF \otimes_0 \Sym{2}{T^*F}}\).
The vanishing of \(s_F\) as a 1-cocycle is of course a weaker condition than requiring that \(s_F\) vanish on each overlap of coordinate charts, as it only requires that \(s_F\) equal the differences on overlaps of a local section on each chart.
We can also consider \(s_F\) as a 1-cocycle in larger sheaves than \(TF \otimes_0 \Sym{2}{TF}^*\).
Molzon and Mortensen \cite{Molzon/Mortensen:1996} p. 3024 construct a holomorphic projective connection explicitly in coordinates out of any local choice of 0-cochain in \(TF \otimes \Sym{2}{TF}^*\) with \(s_F\) as coboundary, but now viewing \(s_F\) as a 1-cocycle in \(TF \otimes \Sym{2}{TF}^*\) rather than in \(TF \otimes_0 \Sym{2}{TF}^*\).
(The same computations work for foliations as for manifolds, so we can make the obvious reinterpretations of the work of Molzon and Mortensen.)
They also prove that the family of holomorphic projective connections on a given complex manifold is parameterized by the 0-cycles in that sheaf.
However, on p. 3019 they show explicitly in coordinates how to alter a 0-cochain in \(TF \otimes \Sym{2}{TF}^*\) with coboundary \(s_F\) to obtain a 0-cochain in \(TF \otimes_0 \Sym{2}{TF}^*\) with the coboundary \(s_F\).
Their construction of a holomorphic projective connection then yields a holomorphic normal projective connection.
Therefore under the splitting 
\[
TF \otimes \Sym{2}{TF}^* = T^*F + \pr{TF \otimes_0 \Sym{2}{TF}^*},
\]
the 1-cocycle \(s_F\) viewed as valued in the left hand side splits into two pieces, but the first vanishes, so \(s_F\) is valued in \(TF \otimes_0 \Sym{2}{TF}^*\).
Warning: Molzon and Mortensen use the term \emph{normal} differently than the definition that we have used, since they don't require that their 0-cochain be traceless. 
We use the same definition as \cite{Calderbank/Diemer:2001,Cap/Schichl:2000} which is the standard definition today.
One sees easily from the explicit expression of Molzon and Mortensen that the obstruction \(s=\pr{s^{\alpha \beta}}\) between the Cartan connections on overlaps defined earlier is exactly the Schwarzian derivative cocycle written explicitly above together with a complicated expression in the first derivative of the Schwarzian derivative.

\begin{lemma}\label{lemma:obstruction.Schwarzian}
The Schwarzian derivative 1-cocycle \(s_F\) of a holomorphic foliation \(F\) on a complex manifold \(M\) vanishes just when the obstruction 1-cocycle \(S(TF)\) obstructing a holomorphic projective connection vanishes.
A holomorphic foliation \(F\) on a complex manifold \(M\) admits a holomorphic projective connection just when it admits a normal holomorphic projective connection, and the space of normal holomorphic projective connections is then a complex affine space modelled on \(\Cohom{0}{M,TF \otimes_0 \Sym{2}{TF}^*}\).
\end{lemma}
\begin{proof}
If this Schwarzian derivative 1-cocycle \(s_F\) vanishes, then there is a holomorphic projective connection constructed by Molzon and Mortensen, and so the obstruction 1-cocycle vanishes.
On the other hand, by the work of Molzon and Mortensen described above, the Schwarzian derivative 1-cocycle \(s_F=s(TF)\) is a quotient of the obstruction 1-cocycle, so if the obstruction 1-cocycle vanishes then the Schwarzian derivative 1-cocycle vanishes as well.
Consequently, a holomorphic normal projective connection appears precisely when a holomorphic projective connection appears.
Once we have one such, say \(\omega\) on \(E \to M\), then any other, say \(\omega'\) on \(E' \to M\), will have its bundle identified, \(E' \cong E\), by adding the requirement that \(\omega'=\omega+S \sigma\) with \(\tortr{S}=0\) as in theorem~\vref{theorem:Difference}.
But to have \(\omega'\) also normal requires precisely the vanishing of the antisymmetric and trace parts of \(S\): see Kobayashi and Nagano \cite{Kobayashi/Nagano:1964} p. 227.
\end{proof}

It is not known how to relate the classical Schwarzian derivative 1-cocycle to the obstruction 1-cocycle of a vector bundle if the vector bundle is not bracket closed.

\subsection{The Atiyah class and projective connections on foliations}

If \(E \to M\) is any holomorphic principal bundle, and \(V \subset TM\) is any holomorphic subbundle, the Atiyah class \(a(M,V,E) \in \Cohom{1}{M,V^* \otimes \ad(E)}\) is the obstruction to the construction of a partial holomorphic connection on \(E\) over \(V\), given explicitly as the cocycle consisting of differences between choices of local holomorphic partial connection.
The Atiyah class \(a(M,V,W)\) of a vector bundle \(W\) is defined to be the Atiyah class of the associated principal bundle.
In particular, if \(V=TF\) is the tangent bundle of a holomorphic foliation, then the Atiyah class of \(a(M,TF,TF)\) is represented by the cochain
\(
\frac{\partial_{jk} Z^{\ell}}{\partial_i Z^{\ell}}
\)
and the Atiyah class of \(a\pr{M,TF,\det T^*F}\) is represented by 
\[
\pd{\log \det \pd{Z}{z}}{z^i}.
\]
Therefore if \(F\) has rank \(p\) then
\[
s_F = a(M,TF,TF) 
- \frac{I}{p+1} \otimes a\pr{M,TF,\det TF^*}
- a\pr{M,TF,\det TF^*} \otimes\frac{I}{p+1},
\]
represents the Schwarzian derivative in more familiar terms for algebraic geometers.
There is no known analoguous interpretation of the Schwarzian derivative of a rank 1 foliation, or of a holomorphic subbundle of \(TM\) which is not bracket closed.

\subsection{Vanishing of all other obstructions}

\begin{proposition}\label{proposition:Schwarzian.is.all}
Suppose that \(M\) is a complex manifold, and \(X=G/P\) a cominiscule variety.
Split up \(X=\prod_i X_i\) into irreducibles, with \(G=\prod_i G_i\) and \(P=\prod_i P_i\).
Then \(M\) has a holomorphic \(\pr{G,X}\)-geometry if and only if \(M\) has
\begin{enumerate}
\item
a holomorphic vector bundle \(V \subset TM\) and
\item
a splitting \(V=\bigoplus_i V_i\) so that the rank of \(V_i\) is the dimension of \(X_i\) and
\item
for each \(X_i\) which is not a projective factor, holomorphic tensors on \(V_i\) pointwise complex-linearly isomorphic to the fundamental tensor and barnacle tensor of \(X_i\) and
\item
for each projective factor \(X_i\), the obstruction 1-cocycle of \(V_i\) vanishes; if it should happen that \(V_i\) is bracket-closed then this occurs just when the classical Schwarzian derivative cocycle of \(V_i\) vanishes.
\end{enumerate}
In particular, if all of the \(V_i\) associated to projective factors \(X_i\) are bracket closed, then there is a holomorphic \(\pr{G,X}\)-geometry inducing the underlying holomorphic \(G_0\)-structure just when there is a \emph{normal} holomorphic \(\pr{G,X}\)-geometry inducing the underlying holomorphic \(G_0\)-structure.
\end{proposition}
\begin{proof}
Once we have the bundle \(V\) and the splitting, the splitting \(X=\prod_i X_i\) splits the obstruction 1-cocycle into a sum \(S = \sum_i S_i\), \(S_i \in \Cohom{1}{M,\Sch_i}\).
If each of these \(S_i\) vanish, then there is a \(\pr{G_i,X_i}\)-geometry on \(V_i\), and if this holds for every \(S_i\), then the product geometry is a \(\pr{G,X}\)-geometry on \(V\).
So without loss of generality, \(X\) is irreducible, and then the existence of an inducing \(\pr{G,X}\)-geometry holds as explained above.

If \(X_i\) is not a projective factor, then we have already ensured that the \(\pr{G_i,X_i}\)-geometry is normal by using lemma~\vref{lemma:normal.space} to construct the geometry.
For each projective factor \(X_i\), if \(V_i\) is bracket-closed, then the existence of a holomorphic projective connection (which we have hypothesized) is equivalent to the vanishing of the obstruction 1-cocycle \(S_i\) by lemma~\ref{lemma:S.obstruction}, which is equivalent to the vanishing of the classical Schwarzian derivative by lemma~\vref{lemma:obstruction.Schwarzian}, which is the obstruction to the existence of a normal holomorphic projective connection (as explained by Molzon and Mortensen \cite{Molzon/Mortensen:1996}).
\end{proof}

\section{Splitting of the tangent bundle}

In this section only we will allow parabolic geometries whose model is not necessarily effective.
If \(P \subset Q \subset G\) are parabolic subgroups of a complex semisimple Lie group \(G\), and \(E \to M\) is a holomorphic \(G/Q\)-geometry, then let \(\hat{M} = E/P\) and then \(E \to \hat{M}\) is a holomorphic \(G/P\)-geometry, called the \emph{lift} of \(E \to M\).
Conversely, we say that a holomorphic \(G/P\)-geometry \emph{drops} to a holomorphic \(G/Q\)-geometry if it is isomorphic to the lift of a \(G/Q\)-geometry.

\begin{theorem}\label{theorem:drop.bracket}
Suppose that \((G,X)\) is a complex homogeneous space, \(G\) acting effectively, holomorphically and transitively.
Suppose that \(G\) preserves a splitting of the tangent bundle of \(X\).
Then clearly any holomorphic \(\pr{G,X}\)-geometry on any complex manifold \(M\) induces a corresponding splitting of the tangent bundle of \(M\).
If \(M\) is a smooth complex projective variety, then the splitting arises from a local product structure on \(M\) or else the geometry drops to a lower dimensional holomorphic Cartan geometry with model \((G,Y)\), where \(X \to Y\) is a \(G\)-equivariant holomorphic fiber bundle.
\end{theorem}
\begin{proof}
The Cartan geometry on \(M\) drops to a lower dimensional Cartan geometry just when \(M\) contains a rational curve \cite{Biswas/McKay:2010a} p. 2 theorem 2.
If a holomorphic splitting on a smooth projective variety \(M\) has a summand which is not bracket closed, then \(M\) is uniruled, i.e. covered by rational curves in an algebraic family \cite{Hoering:2007} p. 466 theorem 1.3, and therefore the geometry drops and the relevant summand in the splitting maps to a new summand in the new splitting.
If the new summand is bracket closed, then so is the original one, and vice versa, since the two geometries have the same curvature tensor.
Therefore the geometry continues to drop until we have eliminated the relevant factor, i.e. it becomes a \(G/G\)-geometry, ensuring that its curvature vanishes, which forces bracket closure in the original geometry.
\end{proof}

\begin{theorem}
Suppose that \(\pr{G,X}\) is a rational homogeneous variety which splits into \(X=\prod_i X_i\), \(G=\prod_i G_i\) into rational homogeneous varieties \(\pr{G_i,X_i}\).
Then clearly any holomorphic \(\pr{G,X}\)-geometry on any complex manifold \(M\) induces a corresponding splitting of the tangent bundle of \(M\).
If \(M\) is a smooth complex projective variety then \(M\) is locally a product of complex manifolds with these subbundles as tangent bundles.
\end{theorem}
\begin{proof}
By theorem~\vref{theorem:drop.bracket} the geometry drops.
The dropped geometry is still a holomorphic parabolic geometry, so we can repeat until all of the subbundles in the resulting splitting are bracket closed, say for a \((G,Y)\)-geometry where \(Y=\prod_i Y_i\), splitting into just as many varieties, except that perhaps some of these \(Y_i\) are not rational homogeneous varieties anymore; this happens precisely when \(Y_i=G_i\).
The original \((G,X)\)-geometry and the resulting \((G,Y)\)-geometry have the same total space and Cartan connection.
For each factor \(Y_i=G_i\), the curvature vanishes and the \((G,X)\)-geometry is a product of a  model geometry \(\pr{G_i,Y_i}\) and a geometry modelled on the remaining factors.
A product geometry has bracket-closed subbundles corresponding to the factors in the product.
So we can assume that \(Y_i \ne G_i\) and that \(\pr{G_i,Y_i}\) is a rational homogeneous variety.
But then the summands in the splitting of the tangent bundle of the original \(\pr{G,X}\)-geometry are each bracket closed clearly just when the same is true for the \(\pr{G,Y}\)-geometry, since they project to those.
\end{proof}

\begin{theorem}
Suppose that \(M\) is a smooth complex projective variety and \(M\) has split tangent bundle \(TM=\bigoplus V_i\).
For factor \(V_i\) of rank 1, there is a holomorphic projective connection on \(V_i\).
In other words the obstruction 1-cocycle of \(V_i\) vanishes.
\end{theorem}
\begin{proof}
Without loss of generality we can assume \(TM=V_1 \oplus V_2\) and that \(V_1\) has rank 1.
If \(V_2\) is bracket closed, then Brunella, Vit\'orio Pereira and Touzet (\cite{Brunella/VitorioPereira/Touzet:2006} theorem 1.1 p. 241) tell us that the universal covering space \(\pi \colon \tilde{M} \to M\) splits as \(\tilde{M}=C \times N\) where \(C\) is a connected and simply connected Riemann surface, with \(\pi^* V_1 = TC\)  and \(\pi^* V_2 = TN\).
We can assume that \(C\) is the unit disk in the complex plane, the complex plane, or the projective line.
In any case, the automorphisms of \(\tilde{M}\) each preserve the standard holomorphic projective connection, so induce a projective connection on \(V_1\).

On the other hand, if \(V_2\) is \emph{not} bracket closed, then the same theorem of Brunella, Vit\'orio Pereira and Touzet shows that \(V_1\) is tangent the to fibers of a holomorphic \(\Proj{1}\)-fiber bundle, for which the existence of a holomorphic projective connection on \(V_1\) is clear.
\end{proof}

\section{Holonomy groups}

\begin{lemma}\label{lemma:FiniteType}
Suppose that \(M\) is a connected complete simply connected K\"ahler manifold.
Write \(M = M_1 \times M_2 \times \dots \times M_s\) as a product of K\"ahler manifolds, \(s \ge 1\), which are themselves not products of lower dimensional K\"ahler manifolds.
Pick a point \(m = \left(m_1, m_2, \dots, m_s\right) \in M\).
Let \(T = T_m M\) and \(T_i = T_{m_i} M_i\).
Then each \(M_i\) has some holonomy group \(H_i\) at \(m_i\), and the holonomy group of \(M\) at \(m\) is \(H = H_1 \times H_2 \times \dots \times H_s\), with \(H_i\) acting trivially on \(T_j\) for \(j \ne i\).
Each group \(H_i\) preserves a complex structure and a Hermitian metric on \(T_i\), so
complexifies to a subgroup \(H_i^{\mathbb{C}}\) of the complex linear group on \(T_i\).
Either
\begin{enumerate}
 \item \(H_i\) is the holonomy group of a unique irreducible compact Hermitian symmetric space or,
 \item \(H_i^{\C{}}\) is of infinite type;
\end{enumerate}
and both possibilities occur precisely for \(H_i=\Un{n}\).
\end{lemma}
\begin{proof}
By Berger's classification of the holonomy groups of Riemannian manifolds \cite{Joyce:2007} p. 53 theorem 3.4.1 the holonomy groups \(H_i\) of complete simply connected K\"ahler manifolds which are not products and not locally symmetric must be among
\[
\Un{n}, \SU{n}, \Sp{n}, \Sp{n} \Un{1}.
\]
Their complexifications are of infinite type. 
Again by Berger's classification, only \(\Un{n}\) occurs as the holonomy of a compact Hermitian symmetric space, and then only for \(\Proj{n}\).
The rest of the result is the deRham splitting theorem \cite{Joyce:2007} p. 47 theorem 3.2.7.
\end{proof}

\section{Cominiscule geometries on K{\"a}hler--Einstein manifolds}

If \(M\) is a compact complex manifold with \(\Chern[1]{M} < 0\) then \(M\) admits a K{\"a}hler--Einstein metric \(g\), which is unique up to constant rescaling \cite{Aubin:1978,Yau:1977,Yau:1978}.
Suppose that \(M\) is a compact complex manifold admitting a K{\"a}hler--Einstein metric.
Let \(T=TM\) be the holomorphic tangent bundle.
Every holomorphic section \(s\) of \(\left(T^* \otimes T\right)^{\otimes}\) is parallel \cite{Bochner/Yano:1953} p. 142, \cite{Kobayashi:1980b} p. 7.

\begin{corollary}\label{corollary:ParallelTensors}
If \(M\) is a compact complex manifold with a K{\"a}hler--Einstein metric and a cominiscule geometry \(E \to M\), then the fundamental tensor, splitting tensors and barnacle tensor of \(E \to M\) are parallel for the K{\"a}hler--Einstein metric.
\end{corollary}

As noticed by Catanese and Di Scala \cite{Catanese/DiScala:2013}, if a holomorphic tensor in  \(\left(T^* \otimes T\right)^{\otimes}\) is isomorphic to a fundamental tensor at just one point of a compact complex manifold \(M\) with \(\Chern[1]{M}<0\), then, since it is parallel, it is isomorphic at every point.
In particular, it determines an underlying holomorphic first order structure, so locally arises from a holomorphic cominiscule geometry.

\begin{proposition}\label{proposition:ConstantSectional}
Suppose that \(M\) is a connected compact K{\"a}hler manifold with \(\Chern{1}<0\) and split tangent bundle \(TM=T_1 \oplus T_2\) and suppose that \(T_1\) has rank \(n_1 \ge 2\) and
let
\[
\eta = 2\pr{n_1+1} \, \Chern[2]{T_1} \, \Chern[1]{T_1}^{n_1-2}
- n_1 \, \Chern[1]{T_1}^{n_1}.
\]
Then \(\eta=0\) just when all of the following occur: 
\begin{enumerate}
\item 
the universal covering space \(\tilde{M}\) of \(M\) splits into a product \(\tilde{M}=M_1 \times M_2\) of complex manifolds and 
\item
the \(T_i\) pull back to the \(T M_i\), and 
\item 
every K\"ahler--Einstein metric on \(M\) pulls back to a a product metric on \(M_1 \times M_2\) and a constant holomorphic sectional curvature metric on \(M_1\).
\end{enumerate}
\end{proposition}
\begin{proof}
Because \(\Chern[1]{M} < 0\)  and \(M\) is K{\"a}hler, \(M\) admits a K{\"a}hler--Einstein metric, unique up to constant rescaling \cite{Aubin:1978,Yau:1977,Yau:1978}.
The holomorphic projections \(T \to T_i \subset T\) are parallel in the K\"ahler--Einstein metric and so the subbundles \(T_i \subset TM\) are parallel. 
By the de~Rham splitting theorem \cite{deRham:1952}, the universal covering space of \(M\) splits into a direct product
\(
\tilde{M} = M_1 \times M_2,
\)
with the metric a local direct product metric, and \(T_i\) pulls back to \(TM_i\) for each \(i\).,
Consequently, the subbundles \(T_i\) are both bracket closed.
The holonomy group of the K\"ahler--Einstein metric reduces to a product \(H=H_1 \times H_2\), with \(H_i\) the holonomy group of \(M_i\).
The complex structure is also a holomorphic tensor and so each \(H_i\) preserves the complex structure on \(T_i\).

Because the splitting is locally a product, the curvature of the K\"ahler--Einstein metric 
vanishes on orthonormal pairs of vectors \(v_1, v_2\) with \(v_i \in T_i\). 
On some open set we pick a local unitary coframing, say \(\omega_i\) for \(T_1\) and a local unitary coframing \(\omega_I\) for \(T_2\), and take the Chern connection \(\omega_{\mu \bar{\nu}}\).
The induced metrics on \(M_1\) and \(M_2\) are K\"ahler--Einstein, since the
Ricci curvature of the restriction of the metric on \(\tilde{M}\) is the restriction of the 
Ricci curvature. 

We will see that \(\eta=0\) just when the sectional curvature of the induced metric on \(T_1\) is constant, as in \cite{Chen/Oguie:1975,Kobayashi/Ochiai:1980,Kobayashi:1983}.
The Ricci tensor has components \(R_{i \barj} = R_{i \barj k \bar k}, R_{I \bar{J}} = R_{I \bar{J} K \bar{K}}\) and all other components vanish.
The scalar curvature of \(M\) is \(s=R_{i \bari} + R_{I \bar{I}}\), while that of \(M_1\) is \(s_1=R_{i \bari}\).
Of course, being K\"ahler--Einstein, \(R_{i \barj} = \lambda \, \delta_{i \barj}, R_{I \bar{J}} = \lambda \, \delta_{I \bar{J}}\).
The scalar curvature is then \(s=\left(n_1 + n_2\right) \lambda\), while \(s_1=n_1 \lambda\).
We write out the Chern forms of the subbundle \(T_1\):
\begin{align*}
\Chern[1]{T_1}
&=
\frac{\sqrt{-1}}{2 \pi}
\nabla \omega_{i \bari}.
\\
&=
\frac{\sqrt{-1}}{2 \pi} 
R_{i \barj} \omega_{\bari} \wedge \omega_j
\\
\Chern[2]{T_1}
&=
-\frac{1}{8 \pi^2}
\left( 
\nabla \omega_{i \bari}  \nabla \omega_{j \barj} 
-
\nabla \omega_{i \barj} \nabla \omega_{j \bari}
\right) 
\end{align*}
The K\"ahler form restricted to \(T_1\) is
\[
\Omega_1 = \frac{\sqrt{-1}}{2 \pi} \omega_{\bari} \wedge \omega_i.
\]
Without even using the K\"ahler--Einstein condition,  
\begin{align*}
\nabla \omega_{i \bari} \wedge \nabla \omega_{j \barj} \wedge \Omega_1^{n_1-2}
&=
\left(
R_{i \bari k \bar{l}}
\omega_{\bar{k}} \wedge \omega_{l}
\right)
\wedge
\left(
R_{j \barj p \bar{q}}
\omega_{\bar{p}} \wedge \omega_{q}
\right)
\wedge
\Omega_1^{n_1-2}
\\
&=
\left( 
  s_1^2 - R_{i \barj} R_{j \bari}
\right) \Omega_1^{n_1}.
\end{align*}
Plugging in the expressions for \(\nabla \omega\) in terms of \(R\),
we find
\begin{align*}
\Chern[1]{T_1}^2 \wedge \Omega_1^{n_1-2}
&=
\frac{1}{4 \pi^2 n_1 \left(n_1 - 1\right)}
\left(
  s_1^2-R_{i \barj} R_{j \bari}
\right)
\Omega_1^{n_1}.
\\
\Chern[2]{T_1} \wedge \Omega_1^{n_1-2}
&=
\frac{1}{8 \pi^2 \, n_1 \left(n_1-1\right)}
\left(
  R_{i \barj k \bar{l}} R_{j \bari l \bar{k}} + s_1^2 - 2 R_{i \barj} R_{j \bari} 
\right)
\Omega_1^{n_1}.
\end{align*}

The holomorphic sectional curvature on \(M_1\) is constant just when
\[
R_{i \barj k \bar{l}} = \frac{s_1}{n_1\left(n_1+1\right)}
\left( 
\delta_{i \barj} \delta_{k \bar{l}} +
\delta_{i \bar{l}} \delta_{k \bar{j}}
\right).
\]
This motivates the definition of a tensor
\[
T_{i \barj k \bar{l}}
=
R_{i \barj k \bar{l}} - \frac{s_1}{n_1\left(n_1+1\right)}
\left( 
\delta_{i \barj} \delta_{k \bar{l}} +
\delta_{i \bar{l}} \delta_{k \bar{j}}
\right).
\]
Computation gives
\[
T_{i \barj k \bar{l}} 
T_{j \bari l \bar{k}} 
=
R_{i \barj k \bar{l}}
R_{j \bari l \bar{k}}  
-
\frac{2 \, s_1^2}{n_1\left(n_1+1\right)}.
\]
In particular, the metric on \(M_1\) has
constant holomorphic sectional curvature just when
\[
 R_{i \barj k \bar{l}}
R_{j \bari l \bar{k}}  
=
\frac{2 \, s_1^2}{n_1\left(n_1+1\right)}.
\]

If we assume next the K\"ahler--Einstein condition, 
that \(R_{i \barj}=\lambda \, \delta_{i \barj}\), then
\[
\Chern[1]{T_1} = \frac{s_1}{2 \pi n_1} \Omega_1,
\]
and
\[
R_{i \barj} R_{j \bari} = \frac{s_1^2}{n_1}.
\]
Plugging this in to our Chern class expressions,
\begin{align*}
\Chern[1]{T_1}^{n_1}
&=
\left(\frac{s_1}{2 \pi n_1}\right)^{n_1} \Omega_1^{n_1}
\\
&=
\left(\frac{s_1}{2 \pi n_1}\right)^{n_1-2}
\Chern[1]{T_1}^2
\wedge
\Omega_1^{n_1-2}
\\
&=
\left(\frac{s_1}{2 \pi n_1}\right)^{n_1-2}
\frac{1}{4 \, \pi^2 \, n_1 \left(n_1-1\right)}
\left(
  s_1^2 - R_{i \barj} R_{j \bari}
\right)
\Omega_1^{n_1}.
\end{align*}
Similarly,
\[
\Chern[2]{T_1} \Chern[1]{T_1}^{n_1-2}
=
\frac{1}{8 \, \pi^2 \, n_1 \left(n_1 - 1 \right)}
\left(\frac{s_1}{2 \pi n_1}\right)^{n_1-2}
\left( 
  R_{i \barj k \bar{l}}
  R_{j \bari l \bar{k}}
  + \frac{n_1-2}{n_1} s_1^2
\right)
\Omega_1^{n_1}.
\]
Therefore 
\begin{align*}
\eta
&=
2\left(n_1+1\right)
\Chern[2]{T_1}
\Chern[1]{T_1}^{n_1-2}
-
n_1 \, \Chern[1]{T_1}^{n_1}
\\
&=
\frac{1}{4 \, \pi^2 \, n_1 \left(n_1-1\right)}
\left(
  \frac{s_1}{2 \pi n_1}
\right)^{n_1-2}
\left(
  (n+1) \, R_{i \barj k \bar{l}}
  R_{j \bari l \bar{k}}
  - \frac{2}{n} s_1^2
\right) \Omega_1^{n_1}
\\
&=
\frac{1}{4 \, \pi^2 \, n_1 \left(n_1-1\right)}
\left(
  \frac{s_1}{2 \pi n_1}
\right)^{n_1-2}
\left(
  (n+1) \, T_{i \barj k \bar{l}}
  T_{j \bari l \bar{k}}
\right) \Omega_1^{n_1}.
\end{align*}
So finally, we see that \(\eta\) is a nonnegative
differential form. In particular, \(\int_M \eta \wedge \Omega_2^{n_2} \ge 0\)
and equality occurs just when \(\eta=0\), which occurs just when \(T=0\),
which occurs just when \(M_1\) has constant sectional curvature.
\end{proof}

\begin{lemma}[McKay \cite{McKay:2007}]\label{lemma:ChernClassRelations}
Suppose that \(G\) is a complex semisimple Lie group and \(P \subset G\) is a parabolic subgroup.
Suppose that \(W_1, W_2, \dots, W_t\) are \(P\)-submodules and \(M\) is a compact K\"ahler manifold bearing a holomorphic \(G/P\)-geometry \(E \to M\).
Let \(V_i = E \times_P W_i\), and also write \(V_i\) for the corresponding vector bundles \(G \times_P W_i\) on the model.
Then every polynomial equation 
\[
0 = p\left(c\left(V_1\right),c\left(V_2\right),\dots,c\left(V_t\right)\right)
\]
in Chern classes satisfied on the model \(G/P\) must be satisfied on \(M\).
\end{lemma}

\begin{corollary}\label{corollary:Chen}
Suppose that \(M\) is a compact K{\"a}hler manifold with \(\Chern{1}<0\), bearing a cominiscule geometry with model 
\[
G/P = \left(G_1/P_1\right) \times \left(G_2/P_2\right) \times \dots \times 
\left(G_s/P_s\right). 
\]
This geometry splits the tangent bundle, say as 
\[
TM = T^1 \oplus T^2 \oplus \dots \oplus T^s.
\]
The universal covering space 
\(
\pi \colon \tilde{M} \to M
\)
splits:
\[
\tilde{M}
=
\tilde{M}_1 \times \tilde{M}_2 \times \dots \times \tilde{M}_s,
\]
so that \(\pi^* T_i = T\tilde{M}_i\). 
Every K{\"a}hler--Einstein metric on \(M\) lifts to a product metric on \(\tilde{M}\),
a product of K{\"a}hler--Einstein metrics on each of these \(\tilde{M}_i\). 
If \(G_i/P_i\) is a projective factor of complex dimension \(2\) or more, then \(\tilde{M}_i\) is complex hyperbolic space. 
\end{corollary}
\begin{proof}
The holonomy group preserves the splitting by corollary~\vref{corollary:ParallelTensors}, so the deRham splitting theorem ensures that the universal covering space splits as indicated.
We could also derive the local splitting from theorem~\vref{theorem:drop.bracket}, because \(\Chern[1]{M}<0\) so \(M\) is a smooth projective variety.
The Chern class inequality in proposition~\vref{proposition:ConstantSectional}
follows from lemma~\vref{lemma:ChernClassRelations}.
Therefore \(\tilde{M}_i\) has a constant negative holomorphic sectional curvature metric, which we can assume has the same sectional curvature as complex hyperbolic space, 
so \(\tilde{M}_i\) must be isomorphic to complex hyperbolic space.
\end{proof}

\section{The Klingler--Mok vanishing theorem}

Suppose that \(G/P\) is a cominiscule variety. 
Decompose \(\LieG\) into simple Lie algebras 
\(
\LieG = 
\LieG^1 \oplus \LieG^2 \oplus \dots \oplus \LieG^N.
\)
For each simple factor \(\LieG^i \subset \LieG\)
let \(\gamma_i\) be the lowest weight whose weight space belongs to \(\LieG^i/\LieP^i=\LieG^i_-\).
This weight \(\gamma_i\) is of course a root, since \(\LieG_-\) is a sum of root spaces of \(\LieG\).
Note that \(\gamma_i\) is \emph{not} the negative noncompact simple root in general.
Let \(\gamma = \sum_i \gamma_i\), a weight (not a root unless \(\LieG\) is simple).
Call \(\gamma\) the \emph{Klingler weight} of \(G/P\).

\begin{theorem}[Klingler \cite{Klingler:2001} p. 212]\label{theorem:KlinglerMok}
Suppose that \(G/P\) is a cominiscule variety with Klingler weight \(\gamma\).
Suppose that \(\lambda\) is the lowest weight of a finite dimensional irreducible \(G_0\)-module \(V\).
Suppose that \(\lambda \ne 0\) and \(\left<\gamma,\lambda\right> \ge 0\).
Suppose that \(M\) is a compact complex manifold whose universal covering space is the noncompact dual Hermitian symmetric space to \(G/P\). 
Take the flat \(G/P\)-geometry \(E \to M\) (as defined in example~\vref{example:dual}.
Then \(0=\Cohom{0}{M,E \times_P V}\).
\end{theorem}
Klinger only states this result when \(G/P\) is an irreducible cominiscule variety; however, his proof works verbatim for any cominiscule \(G/P\).
(His sign conventions differ from ours, so his statement looks a little different.)


\section{Holonomy and cominiscule geometries}

\begin{theorem}\label{theorem:big}
Suppose that \(M\) is a connected compact complex manifold with \(\Chern{1}<0\) bearing a cominiscule geometry.
Then \(M\) admits a normal cominiscule geometry, and every normal cominiscule geometry is standard.
All (either normal or abnormal) cominiscule geometries on \(M\) have the same model and the  same holomorphic principal bundle \(E \to M\) with the same holomorphic map \(E \to FM\) to the underlying first order structure.
If the geometry has Cartan connection \(\omega\), then each cominiscule geometry has a Cartan connection of the form \(\omega'=\omega + s \omega_-\) for some unique \(s \in \Cohom{0}{M,\Sch}\).
Thus the moduli space of cominiscule geometries on \(M\) is a finite dimensional complex vector space, canonically identified with \(\Cohom{0}{M,\Sch}\).
\end{theorem}
\begin{proof}
Suppose that the geometry has model \(X=G/P\).
Split the model into irreducibles as \(X=\prod_i X_i\), \(G=\prod_i G_i\), \(P=\prod_i P_i\).
The universal covering space \(\tilde{M}\) of \(M\) has split tangent bundle.
The splitting is invariant parallel transport by corollary~\vref{corollary:ParallelTensors}.
By the deRham splitting theorem~\cite{Joyce:2007} p. 47 theorem 3.2.7, the universal covering \(\tilde{M}\) of \(M\) then splits as a corresponding product
\(
\tilde{M} = \prod_i \tilde{M}_i.
\)
We can assume that the geometry is normal, by proposition~\vref{proposition:Schwarzian.is.all}.
Each factor \(\tilde{M}_i\) inherits a normal cominiscule geometry with irreducible model, invariant under parallel transport.

If the factor \(X_i\) is not projective, the holonomy group is the holonomy group of the \(G_i'\)-invariant metric on the noncompact Hermitian symmetric space, \(X'_i\) and \(\tilde{M}_i=X'_i\).
Moreover, the underlying first order structure is invariant under parallel transport. 
The complexified holonomy group is precisely the stabilizer of that first order structure, and also precisely the stabilizer of the standard first order structure.
The proof of lemma~\vref{lemma:IdentityComponent} easily adapts to prove that the data of the first order structure (the fundamental tensor, barnacle tensor and splitting) are determined completely by the action of the group \(G_0\), so that any two first order structures with the same stabilizer are identical, ensuring the uniqueness of the first order structure on \(\tilde{M}_i\).
Klingler \cite{Klingler:2001} p. 211 propositions 4.9 and 4.10 gives an alternative proof of the rigidity of the underlying first order structure.

If the factor \(\pr{G_i,X_i}\) is a projective factor, say \(X_i=\Proj{n}\) for \(n \ge 2\), then corollary~\vref{corollary:Chen} and lemma~\vref{lemma:ChernClassRelations} together prove that \(\tilde{M}_i\) is complex hyperbolic space and the metric on \(\tilde{M}\) is a product metric.
We need to establish that the projective connection on \(\tilde{M}_i\) is the standard flat one. 
As shown in lemma~\vref{lemma:obstruction.Schwarzian}, given any two normal holomorphic projective connections on \(M\), one can be written in terms of the other and a holomorphic section of \(T^i \otimes_0 \Sym{2}{T^i}^*\), the bundle of traceless symmetric vector-valued 2-tensors, after a suitable holomorphic bundle isomorphism.
(This is also noted with less detail by Klingler \cite{Klingler:2001} p. 17).
It follows from theorem~\vref{theorem:KlinglerMok} that there are no global nonzero sections of this bundle; see Klingler \cite{Klingler:2001} p. 17 for more details. 
Therefore the normal holomorphic projective connection on \(\tilde{M}_i\) is the usual flat one on complex hyperbolic space, and this is the unique holomorphic normal projective connection on \(T^i\).

If \(X_i=\Proj{1}\) then \(\tilde{M}_i\) inherits a holomorphic projective connection given by a holomorphic quadratic differential, say \(\eta_i\).

The product of all of these geometries on these various factors is a normal holomorphic cominiscule geometry with the same underlying first order structure as the pullback geometry, and inducing the same projective connections on the various vector subbundles.
If there are two holomorphic normal \(\pr{G,X}\)-geometries on \(M\) with the given underlying first order structure and inducing the same holomorphic projective connections on each \(T^i\) with rank \(1\), then they must each induce the same normal geometries on each of the \(T^i\), and so the Cartan connection must restrict to the same Cartan connection above each of these \(T^i\), on the same holomorphic bundle (corollary~\vref{corollary:glue}), forcing them to agree.
Therefore the product geometry is the unique normal holomorphic cominiscule geometry on \(M\) with the given model, and the geometry on \(M\) is flat.
We still need to prove that it is standard, i.e. that we can replace \(M\) by a finite covering space to split \(M\) into a product of compact Riemann surfaces with holomorphic quadratic differentials and some locally Hermitian symmetric variety with its standard cominiscule geometry.
If \(\dim M = 1\), then the theorem is clear, so we can assume that \(G\) has rank at least \(2\).

We call \(\Gamma\) \emph{reducible} if \(\Gamma\) has a finite index subgroup \(\Gamma' \times \Gamma''\) with each of \(\Gamma'\) and \(\Gamma''\) lying in different factors in the product \(G=\prod_i G_i\). 
Consequently, some finite cover of \(M\) splits into a corresponding product \(M=M' \times M''\), and the product geometries on the universal covering space of \(M\) then descend to products of geometries on each factor and our proof is finished by induction.
So we can assume that \(\Gamma\) is irreducible without loss of generality.
The problem arises that there might exist some projective line factor \(X_i=\Proj{1}\) and some nonzero quadratic differential \(\eta_i\) on the associated factor \(\tilde{M}_i\).

We follow Klingler \cite{Klingler:2001} p. 205.
Suppose that there is such a quadratic differential, invariant under \(\Gamma\).
The set of \(\Gamma\)-invariant quadratic differentials is a complex vector space.
We have only to prove that it has dimension zero.
The geometry on \(M\) is flat, so there is a developing map \(\delta \colon \tilde{M} \to X\) to the model and a morphism \(h \colon \pi_1(M) \to G\) so that \(\delta \circ \gamma = h(\gamma)\delta\) for all \(\gamma \in \pi_1(M)\).
By the Ehresmann--Thurston deformation theorem \cite{Goldman:2010}, the complex analytic stack of flat \(\pr{G,X}\)-geometries on the real manifold \(M\) is locally isomorphic to the complex analytic stack of representations of the fundamental group, by taking each geometry to its holonomy morphism \(h\).
The local rigidity of the flat holomorphic projective connection therefore follows if we can prove the local rigidity of the representation \(h\).
A theorem of Weil tells us that the vanishing of \(\Cohom{1}{\Gamma,\LieG}\) implies local rigidity of this representation \cite{Raghunathan:1972} p. 91 theorem 6.7.
Weil's rigidity theorem \cite{Raghunathan:1972} p. 137 theorem 7.66 implies the vanishing of \(\Cohom{1}{\Gamma,\LieG}\), and therefore the normal geometry on \(M\) is locally rigid.
The \(\Gamma\)-invariant quadratic differentials on the various 1-dimensional factors are the deformations of the geometry on \(M\) among normal geometries.
Since there are no deformations, this vector space has dimension zero: the geometry is the unique normal holomorphic projective connection on \(M\).

There are no cominiscule geometries on \(M\) with different models, because the model is the compact Hermitian symmetric space which is the compact form of the universal covering space of \(M\), with the identity component of its biholomorphism group acting on it. 
The identification of the moduli space follows by lemma~\vref{lemma:normal.space}.
\end{proof}

\section{Conclusion}
Our results above generalize those of \cite{Beauville:2000,Klingler:2001,Kobayashi/Ochiai:1980,Kobayashi/Ochiai:1981,Kobayashi/Ochiai:1981b,Kobayashi/Ochiai:1982}.
The dimension of the moduli space of holomorphic cominiscule (not necessarily normal) geometries on a general K\"ahler--Einstein manifold is still unknown, although we have identified it with the space of holomorphic sections of \(\Sch\).
It seems unlikely that a compact complex manifold with \(\Chern{1}<0\) can admit a \emph{noncominiscule} holomorphic parabolic geometry.
Smooth projective curves, surfaces and 3-folds admitting holomorphic projective connections on are classified by Klingler \cite{Klingler:1998} and Jahnke and Radloff \cite{Jahnke/Radloff:2011}; besides the model, the translation invariant examples on tori, and the locally complex hyperbolic spaces (as above), there are also certain torus bundles over curves.
It should be feasible to classify holomorphic projective connections on torus bundles over K\"ahler--Einstein manifolds in any dimension.

\bibliographystyle{amsplain}
\bibliography{KE-CHSS}
\end{document}